\documentclass{article}
\usepackage{graphicx} 
\usepackage{amsmath} 
\allowdisplaybreaks 
\usepackage{amssymb}  
\usepackage{amsthm,thmtools,dsfont} 
\usepackage{tabu} 
\usepackage[colorlinks=true,
            linkcolor=black,
            citecolor=black,
            urlcolor=teal]{hyperref} 
\usepackage{cleveref} 
\usepackage{enumitem} \setlist[itemize,1]{leftmargin=\dimexpr 20pt} 
\usepackage{cite} 
\usepackage{microtype,enumerate,url} 
\usepackage{BOONDOX-ds} 
\usepackage{tikz} 
\usepackage[fulladjust]{marginnote} 
\usepackage{xstring} 
\usepackage{booktabs} 
\usepackage{colortbl} 
\usetikzlibrary{arrows,patterns} 
\usetikzlibrary{decorations.pathmorphing}
\usepackage{caption} 
\usepackage{mathabx} 
\usepackage{mathrsfs} 
\usepackage{bbm}

\usepackage{sectsty}
\usepackage{titlesec}
\usepackage{etoolbox}
\allsectionsfont{\sffamily}
\makeatletter
\DeclareRobustCommand{\textbf}[1]{%
  \begingroup
    \bfseries\sffamily #1%
  \endgroup
}
\makeatother
\renewenvironment{abstract}
  {\begin{center}\normalfont\sffamily\bfseries Abstract\end{center}\begin{quote}\rmfamily}
  {\end{quote}}

\usepackage{textcomp}
\def\BibTeX{{\rm B\kern-.05em{\sc i\kern-.025em b}\kern-.08em
    T\kern-.1667em\lower.7ex\hbox{E}\kern-.125emX}}
\usepackage{fontawesome}

\definecolor{hotPink}{RGB}{255,105,180}     
\definecolor{coolTeal}{RGB}{0,160,170}      
\definecolor{goldenRod}{RGB}{255,180,0}     

\newcommand{\ie}{\textit{i.e\@.}}

\makeatletter
\def\@IEEEsectpunct{.\ \,}
\def\paragraph{\@startsection{paragraph}{4}{\z@}{1.5ex plus 1.5ex minus 0.5ex}%
{0ex}{\normalfont\normalsize\itshape}}
\makeatother

\declaretheorem[style=definition]{theorem}
\declaretheorem[style=definition]{corollary}

\declaretheorem[style=definition]{lemma}
\declaretheorem[style=definition,qed=$\vartriangle$]{remark}
\declaretheorem[style=definition,qed=$\blacktriangle$]{example}
\declaretheorem[style=definition,numbered=no]{standing assumption}

\newcommand {\nn}{\nonumber}
\newcommand{\beq}{\begin{equation}}
\newcommand{\eeq}{\end{equation}}
\newcommand {\bseq}{\begin{subequations}}
\newcommand {\eseq}{\end{subequations}}
\newcommand {\bma}{\left[}
\newcommand {\ema}{\right]}

\newcommand{\bsm}{\left[\begin{smallmatrix}}
\newcommand{\esm}{\end{smallmatrix}\right]}

\newcommand {\N}{\mathbb{N}} 	
\newcommand {\Z}{\mathbb{Z}} 	
\newcommand {\R}{\mathbb{R}} 	

\newcommand {\T}{\mathbb{T}} 	
\newcommand {\W}{\mathbb{W}} 	
\newcommand {\B}{\mathcal{B}} 	
\newcommand {\A}{\mathcal{A}} 	
\renewcommand{\L}{\mathcal{L}} 
\newcommand{\syz}{\mathbf{syz}\,} 
\newcommand{\wo}{w_{\textup{o}}} 

\newcommand{\Ker}{\mathbf{ker}} 
\newcommand{\Image}{\mathbf{im}} 
\newcommand{\rank}{\mathbf{rank}} 
\newcommand{\transpose}{\mathsf{T}} 

\newcommand{\col}{\operatorname{col}}
\newcommand{\diag}{\operatorname{diag}}

\usepackage[acronym]{glossaries}
\makeglossaries
\newacronym{PE}{PE}{\text{P}ersistence of  \text{E}xcitation}
\newacronym{LTI}{LTI}{\textbf{L}inear \textbf{T}ime-\textbf{I}nvariant}
\newacronym{ATI}{ATI}{\textbf{A}ffine \textbf{T}ime-\textbf{I}nvariant}
\newacronym{LQT}{LQT}{\textbf{L}inear \textbf{Q}uadratic \textbf{T}racking}
\newacronym{LQR}{LQR}{\textbf{L}inear \textbf{Q}uadratic \textbf{R}egulator}
\newacronym{LPV}{LPV}{\textbf{L}inear \textbf{P}arameter \textbf{V}arying}
\newacronym{DeePC}{DeePC}{\textbf{D}ata-\textbf{e}nabl\textbf{e}d \textbf{P}redictive \textbf{C}ontrol}

\begin{document}
\title{\textsf{\textbf{From Time Series to Affine Systems}}}

\author{
A.~Padoan\thanks{A.~Padoan is with the Department of Electrical and Computer Engineering, University of British Columbia, Vancouver, BC, Canada 
(e-mail: \texttt{alberto.padoan@ubc.ca}).}, %
J.~Eising\thanks{J.~Eising is with the Department of Information Technology and Electrical Engineering, ETH~Zürich, 8092~Zürich, Switzerland 
(e-mail: \texttt{jeising@control.ee.ethz.ch}).}, %
I.~Markovsky\thanks{I.~Markovsky is with the Catalan Institution for Research and Advanced Studies, 08010~Barcelona, Spain, 
and with the International Centre for Numerical Methods in Engineering, 08034~Barcelona, Spain 
(e-mail: \texttt{ivan.markovsky@cinme.upc.edu}).}%
}

\maketitle

\begingroup
\renewcommand\thefootnote{}
\footnotetext{
Manuscript submitted for review, October~2025. 
This work was supported in part by the Swiss National Science Foundation under the National Centre of Competence in Research (NCCR) on Dependable Ubiquitous Automation (Grant~No.~51NF40\_180640), 
and in part by the Natural Sciences and Engineering Research Council of Canada (NSERC) through a Discovery Grant (RGPIN-2025-04577).
}
\addtocounter{footnote}{-1}
\endgroup

\begin{abstract}%
\noindent
The paper extends core results of behavioral systems theory from linear to affine time-invariant systems. We characterize the  behavior of affine time-invariant systems via kernel, input–output, state-space, and finite-horizon data-driven representations,  demonstrating a range of structural parallels with linear time-invariant systems.  Building on these representations, we introduce a new persistence of excitation condition tailored to the model class of affine time-invariant systems. The condition yields a new fundamental lemma that parallels the classical result for linear systems, while provably reducing data requirements. Our analysis highlights that excitation conditions must be adapted to the model class: overlooking structural differences may lead to unnecessarily conservative data requirements.
\end{abstract}

\section{Introduction}

Data-driven control is a powerful paradigm for designing controllers directly from measured trajectories~\cite{markovsky2021behavioral}, without explicitly passing through the intermediate step of identifying a parametric model. This approach finds a natural foundation in behavioral systems theory~\cite{willems1986timea,willems1986timeb,willems1987timec}, which   models systems by their \textit{behavior},  i.e., the set of all its trajectories, rather than by a specific  representation. A key idea underpinning many developments in this area is that all trajectories of a \gls{LTI} system over a finite horizon can be parameterized by a data matrix built from a single sufficiently rich experiment. The richness condition is formalized by the classical notion of \emph{persistence of excitation}~\cite{soderstrom1989system,ljung1999system,VanOverschee1996subspace}, namely, the requirement that the Hankel matrix constructed from the input sequence has full row rank. A central result in this context is the \emph{fundamental lemma}\cite{willems2005note,vanWaarde2020willems}, which states that for a controllable \gls{LTI} system of order $n$, if the input is persistently exciting of order $n+L$, then every trajectory of length $L$ can be expressed as a linear combination of columns of the Hankel matrix built from the measured data. This result provides a bridge from data to trajectories, enabling finite-horizon representations of system dynamics using raw data alone, and forms the basis for a broad range of identification, prediction, filtering, and control algorithms; see\cite{markovsky2021behavioral} for a recent overview.
In the last decade, the fundamental lemma has fueled a resurgence of interest in \textit{direct} data-driven control frameworks~\cite{de2019formulas,vanWaarde2020willems,coulson2019data}, where the goal is to infer optimal decisions directly from measured data rather than through system identification.  Building on this principle, a rich class of data-driven \textit{predictive} control schemes has emerged~\cite{markovsky2021behavioral}. A prominent example is the \gls{DeePC} algorithm~\cite{coulson2019data}, which has been successfully applied in a variety of experimental settings, including aerial robotics~\cite{coulson2021distributionally}, synchronous motor drives~\cite{carlet2020data}, and grid-connected power converters~\cite{huang2019data}, among others.  Similar ideas have since been extended beyond the linear deterministic setting, leading to data-driven formulations of predictive control tailored to nonlinear systems~\cite{berberich2022linear,Lazar2024ECC_BasisFunctionsDeePC} and stochastic  systems~\cite{breschi2023datadriven,Faulwasser2023ARC}.  

Despite these successes, the scope of the fundamental lemma, and the data-driven methods that build upon it, crucially rely on a \textit{linearity} assumption on the underlying system. Yet, a range of recently developed data-driven control~\cite{berberich2022linear,martinelli2022data,vankan2025federated} and identification~\cite{Markovsky2025Affine} algorithms relax the linearity assumption and instead rely on \textit{affine} systems, whose behaviors are affine spaces of trajectories and therefore not captured by the fundamental lemma. Aside from arising in many applications, such as the error dynamics of a reference tracking problem, affine systems arise naturally also through \textit{linearization} a nonlinear system around an operating point that is not necessarily an equilibrium or when accounting for nonzero offsets in the dynamics~\cite{Markovsky2025Affine}. 

The difference is structural: the behavior of an \gls{LTI} system is a \textit{linear}, shift-invariant subspace, closed under linear combinations of trajectories. In contrast, an \gls{ATI} system yields an \textit{affine}, shift-invariant subspace, closed only under affine combinations. This subtle difference has concrete implications: the classical persistence of excitation condition, which ensures data informativity for \gls{LTI} systems via rank conditions on Hankel matrices, cannot be applied \textit{verbatim} to affine systems. The notion of excitation itself must be redefined to reflect the underlying model class.
This reveals a broader conceptual gap: the classical notion of persistence of excitation is \textit{not} a universal measure of richness but one intrinsically tailored to \gls{LTI} behaviors. It guarantees that a time series and its shifts span a linear trajectory subspace, yet ignores structural information present in other model classes.
Extensions  beyond \gls{LTI} systems thus require redefining excitation, so as to reflect the structure of the model class or the input-generation mechanism, a direction already explored in~\cite{Padoan2023DataDriven,Padoan2016CDC,Padoan2016NOLCOS,padoan2017geometric}.  For affine systems, one needs to design experiments that simultaneously excite the ``homogeneous'' dynamics and reveal a particular ``offset'' trajectory. Extensions beyond the linear case include robust and data-efficient formulations~\cite{Berberich2023ACC,CamlibelRapisarda2024}, data-driven approaches for \gls{LPV} and bilinear systems~\cite{Verhoek2023DirectLPV,lpv-fl,lpv-ident},  and  classes of nonlinear systems~\cite{Padoan2018ECC,Berberich2019TrajectoryFramework,Straesser2021BeyondPoly}.

In this paper, we make this structural perspective precise. We first develop a complete characterization of affine time-invariant systems, including their kernel, input–output, state-space, and data-driven representations, together with integer invariants and controllability properties. Building on this foundation, we introduce a new persistence of excitation condition tailored to the model class of \gls{ATI} systems and establish a corresponding fundamental lemma.
Our results provide (i) a complete behavioral characterization of affine representations, (ii) a structurally justified and strictly weaker excitation condition, and (iii) an extension of the fundamental lemma that reduces data requirements compared to existing alternatives in the literature. Finally, we illustrate the practical implications in the context of data-driven nonlinear predictive control, where local affine approximations naturally arise and our framework enables data-efficient system representations.

\textit{Related work.} The fundamental lemma has been originally stated using polynomial algebra~\cite{willems2005note} and later reformulated for state-space systems~\cite{vanWaarde2020willems}. It has enabled a broad range of direct data-driven frameworks, including explicit solutions to the \gls{LQT} problem~\cite{markovsky2008data}, Lyapunov-based control design~\cite{de2019formulas}, and methods relying on quadratic inequalities~\cite{vanWaarde2023qmi}. In parallel, a rich class of data-driven predictive control schemes has emerged~\cite{markovsky2021behavioral}, with extensions to both deterministic~\cite{berberich2022linear,coulson2019data} and stochastic settings~\cite{breschi2023datadriven}; see~\cite{markovsky2021behavioral} for an overview.  
For affine systems, first extensions of the fundamental lemma have been developed in~\cite{martinelli2022data,berberich2022linear}. The results in~\cite{martinelli2022data,berberich2022linear} show that controllability, together with a classical persistence of excitation assumption, ensures that the finite-horizon behavior of an \gls{ATI} system can be parameterized by the column space of a Hankel matrix constructed from a trajectory of the system and augmented with an additional bottom row of ones. A subsequent contribution~\cite{Padoan2023DataDriven} has later established that finite-horizon affine behaviors can be characterized under a generalized excitation condition, without requiring controllability. The representation determined by the fundamental lemma for \gls{ATI} systems has found applications in data-driven nonlinear control~\cite{berberich2022linear,vankan2025federated}, where predictive control algorithms rely on local affine approximations of the underlying nonlinear system to perform tracking tasks. The same representation has also been used in system identification~\cite{Markovsky2025Affine}, where exploiting the structure of affine models yields algorithms that outperform the classical two-step method based on data centering~\cite{ljung1999system}. 
To the best of our knowledge, the works most closely related to the present paper are~\cite{martinelli2022data,berberich2022linear,Padoan2023DataDriven}. In this work, we demonstrate that the persistence of excitation conditions given therein are stronger than necessary, which translates into unnecessarily demanding data requirements to ensure the fundamental lemma holds for \gls{ATI} systems.

\textit{Contributions.} 
We develop a complete behavioral characterization of \gls{ATI} systems, motivated by excessive data requirements of existing data-driven control algorithms that rely on the classical notion of persistence of excitation. First, we establish new representation theorems for discrete-time \gls{ATI} systems, showing that every shift-invariant affine behavior admits a family of equivalent  representations. These results form the foundation for analyzing controllability of \gls{ATI} systems and for defining their integer invariants. Second, we introduce a new persistence of excitation condition tailored to the model class of \gls{ATI} systems. Building on ideas presented in~\cite{Padoan2023DataDriven}, we then develop a generalization of the fundamental lemma for \gls{ATI} systems~\cite{vanWaarde2020willems}, which leverages a rank conditions under provably weaker assumptions and reduced data requirements compared to~\cite{martinelli2022data,berberich2022linear,Padoan2023DataDriven}. Finally, we establish a consistency result showing that, in contrast to the \gls{LTI} case, some affine representations may define empty trajectory sets. In summary, our main contributions are threefold: (i) a complete characterization of \gls{ATI} system representations; (ii) a new notion of persistence of excitation condition; and (iii) a generalization of the fundamental lemma with reduced data requirements.

\textit{Paper organization.} Section~\ref{sec:motivation} provides a motivating application arising in predictive control. Sections \ref{sec:affine_systems} and~\ref{sec:props} introduce the necessary background on behavioral systems theory and characterizes affine systems using this framework. Section~\ref{sec:fundamental_lemma_ATI} presents a new persistence of excitation condition tailored to \gls{ATI} systems and establishes a generalization of the fundamental lemma. Section~\ref{sec:discussion} highlights differences with the classical persistence of excitation, detailing the corresponding data requirements and geometric implications. Section~\ref{sec:example} illustrates our results with a worked-out example. Section~\ref{sec:conclusion} concludes the paper with a summary and an outlook on future research directions.  
The appendix contains all proofs and a digression on equivalence and consistency of kernel representations.

\textit{Notation and terminology.}
The sets of real,  complex,  integer   and positive integer numbers are denoted by  $\R$,  $\mathbb{C}$,
 $\Z$,    and $\N$, respectively.  For ${T\in\N}$, the set of integer numbers $\{1, 2, \dots , T\}$ is denoted by $\mathbf{T}$.  
A function $f$ from $X$ to $Y$ is denoted by ${f:X \to Y}$; $(Y)^{X}$ denotes the set of all such maps. 
The \textit{restriction} of ${f:X \to Y}$ to a set ${X^{\prime}}$, with ${X^{\prime} \cap X \neq \emptyset}$, is denoted by $f|_{X^{\prime}}$ and is defined by $f|_{X^{\prime}}(x)$ for ${x \in X^{\prime}}$; if ${\mathcal{F} \subseteq (Y)^{X}}$, then $\mathcal{F}|_{X^{\prime}}$ is defined as ${\{ f|_{X^{\prime}} \, : \, f \in \mathcal{F}\}}$. 
 A {sequence} is a function \( w \) defined on a non-empty subset  \(\T \subseteq \Z\), often denoted by \( \{w_t\}_{t \in \T} \). The sequence \( w \) is \emph{infinite} if \( \T \) has infinite cardinality and \emph{finite} otherwise, in which case it is often identified with the column vector  ${w = \col(w(1), \ldots, w(|\T|))}$. The \textit{length} of a finite sequence is the cardinality of its domain. The shift operator ${\sigma}$ maps a sequence $w$ with domain \( \T \) to the sequence $\sigma w$ defined as ${(\sigma w) (t)= w (t+1)}$ for all $t\in\T$; if ${\mathcal{W}}$ is a family of sequences, then ${\sigma \mathcal{W} = \{\sigma w \, | \,  w\in\mathcal{W} \}}$.   The transpose, rank, image, and kernel of a matrix ${M\in\R^{m\times n}}$ are denoted by $M^{\top}$, $\rank\, M$, $\Image\, M$, and $\Ker M$, respectively. The $m \times p$ zero matrix is denoted by $0_{m \times p}$ and the $m \times p$ matrix with ones on the main diagonal and zeros elsewhere is denoted by $I_{m \times p}$. For $m = p$, we use $0_m$ and $I_m$. The column vector in $\R^m$ with all entries equal to one is denoted by $\mathbb{1}_m$. Subscripts are omitted if dimensions are clear from context. $\R^{m\times n}[\xi]$ is the set of $m \times n$ matrices with polynomial entries indeterminate $\xi$ and real coefficients. If $R\in\R^{m\times n}[\xi]$ is a polynomial matrix, then $\deg(R)$ denotes its degree.

\section{\!\! A motivating application in data-driven control}  \label{sec:motivation}

Consider a discrete-time system described by the  equations   
\beq \label{eq:system_nonlinear}
\sigma x = f(x, u), 
\quad
y = h(x, u),     
\eeq
\noindent
where $x(t) \in \mathbb{R}^n$, $u(t) \in \mathbb{R}^m$, and $y(t) \in \mathbb{R}^p$ denote the state, input, and output of the system at time $t \in \mathbb{Z}$, respectively.  Given $\bar{x} \in \mathbb{R}^n, $ $ \bar{u} \in \mathbb{R}^m,$ $ \bar{y} \in \mathbb{R}^p$ and assuming that $f$ and $h$ are differentiable at $(\bar{x}, \bar{u})$, the \textit{linearized system} associated with~\eqref{eq:system_nonlinear} around the point $(\bar{x}, \bar{u}, \bar{y})$ is defined as
\beq \label{eq:system_nonlinear_linearization_affine}
\sigma \xi = A \xi + B v + E, 
\quad
z = C \xi + D v + F,
\eeq
where $\xi(t) \in \mathbb{R}^n$, $v(t) \in \mathbb{R}^m$, and $z(t) \in \mathbb{R}^p$ are deviations from the operating points $\bar{x},$ $ \bar{u}, $ and $\bar{y}$, defined as
$$
\xi(t) := x(t) - \bar{x}, \quad v(t) := u(t) - \bar{u}, \quad z(t) := y(t) - \bar{y}, 
$$
and $A\in\R^{n\times n}$, $B\in\R^{n\times m}$, $C\in\R^{p\times n}$, $D\in\R^{p\times m}$, $E\in\R^{n}$, and $F\in\R^{p}$ are the system matrices, defined as
\begin{subequations} \label{eq:system_nonlinear_linearization_affine_matrices}
\begin{align}
    A &:= \frac{\partial f}{\partial x} \bigg|_{(\bar{x}, \bar{u})}, \quad 
    B := \frac{\partial f}{\partial u} \bigg|_{(\bar{x}, \bar{u})}, \quad 
    E := f(\bar{x}, \bar{u}) - \bar{x}, \\
    C &:= \frac{\partial h}{\partial x} \bigg|_{(\bar{x}, \bar{u})}, \quad 
    D := \frac{\partial h}{\partial u} \bigg|_{(\bar{x}, \bar{u})}, \quad 
    F := h(\bar{x}, \bar{u}) - \bar{y}.
\end{align}
\end{subequations}
\noindent
By construction, the linearized system~\eqref{eq:system_nonlinear_linearization_affine} is \textit{affine} and \textit{time-invariant}, i.e., defined by affine functions of the state and input variables with coefficients that do not depend on time. In the special case where $(\bar{x}, \bar{u}, \bar{y})$ is an equilibrium point of the original system~\eqref{eq:system_nonlinear}, i.e., $f(\bar{x}, \bar{u}) = \bar{x}$, and $h(\bar{x}, \bar{u}) = \bar{y}$, the offsets $e$ and $r$ are identically zero, and the linearized system~\eqref{eq:system_nonlinear_linearization_affine} reduces to a \textit{linear} system described by the equations
\[ 
\sigma \xi = A \xi + B v, 
\quad
z = C \xi + D v.
\]
\noindent Linearization has long served as a foundational tool in nonlinear systems theory~\cite{khalil1996nonlinear,slotine1991applied,isidori1995nonlinear}, underpinning a broad range of analysis and design methods that rely on models derived from first principles or through system identification~\cite{ljung1999system}. In many applications, however, obtaining such models is not always possible or even desirable. The system may be too complex to model from first principles, or data for system identification may be inadequate. Recently, these limitations have motivated growing interest in non-parametric system representations~\cite{de2019formulas,vanWaarde2020willems,coulson2019data}, particularly in the context of predictive control~\cite{coulson2019data,berberich2022linear}. Rather than obtaining classical parametric representations, such as state-space or transfer function representations, the idea is to characterize finite-horizon system behaviors using raw data matrices~\cite{coulson2019data,berberich2022linear}, without \textit{a priori} complexity constraints.  

The recent papers~\cite{berberich2022linear,vankan2025federated} build on this idea by proposing a linear tracking predictive control scheme for nonlinear systems. For simplicity, we focus on the formulation presented in~\cite{berberich2022linear}. The method hinges on two key ingredients: first, linearizing the nonlinear system~\eqref{eq:system_nonlinear} around a given trajectory; second, representing the finite-horizon behavior of resulting affine system~\eqref{eq:system_nonlinear_linearization_affine}-\eqref{eq:system_nonlinear_linearization_affine_matrices}  using raw data matrices.   To formalize this idea, we recall~\cite[Theorem~1]{berberich2022linear}. 

Let ${T} \in \N$. The \textit{Hankel matrix} of depth ${L\in\mathbf{T}}$ associated with the sequence ${u\in {(\R^m)}^\mathbf{T}}$ is defined as 
\beq \label{eq:Hankel} 
\! H_{L}(u) \! = \!
\scalebox{1}{$
\bma  \nn
\begin{array}{ccccc}
u(1) & u(2)  & \cdots &  u(T-L+1)   \\
u(2) & u(3)  & \cdots &   u(T-L+2)   \\
\vdots  & \vdots  & \ddots & \vdots  \\
u(L) & u(L+1)  & \cdots  & u(T)
\end{array}
\ema
$} . \! \!
\eeq

\begin{theorem}~\cite[Theorem~1]{berberich2022linear} \label{thm:berberich22}
Consider the system~\eqref{eq:system_nonlinear_linearization_affine}. Let 
\(u_d \in (\R^m)^{\mathbf{T}}\),
\(x_d \in (\R^n)^{\mathbf{T}}\),
\(y_d \in (\R^p)^{\mathbf{T}}\)  be an input/state/output trajectory of system~\eqref{eq:system_nonlinear_linearization_affine} and let ${L\in\mathbf{T}}$. Assume 
\beq \label{eq:rank_condition_affine}
    \rank \, 
    \scalebox{1}{$
    \bma  
    \begin{array}{c}
    H_1(x_d) \\
    H_L(u_d) \\ 
    \mathbb{1}^{\transpose} 
    \end{array}
    \ema = mL + n + 1$} .
\eeq
Then every trajectory $\left[\begin{smallmatrix} u \\ y\end{smallmatrix}\right] \in \R^{(m+p)\mathbf{L}}$ of system~\eqref{eq:system_nonlinear_linearization_affine} of length \(L\) can be expressed as
    \beq \label{eq:image_condition}
    \begin{bmatrix}
    u \\
    y
    \end{bmatrix}
    =
    \begin{bmatrix}
    {H}_L(u_{d}) \\
    {H}_L(y_{d})
    \end{bmatrix} g,
    \eeq
for some \({g \in \R^{T-L+1}}\) such that ${\mathbb{1}^{\transpose} g = 1}$. 
\end{theorem}

\noindent
Theorem~\ref{thm:berberich22} hinges on the rank condition~\eqref{eq:rank_condition_affine}, which ensures that the data matrices 
${H_1(x_d)}$ and ${H_L(u_d)}$ contain sufficiently rich information to parameterize all trajectories of the affine system~\eqref{eq:system_nonlinear_linearization_affine} over a finite horizon. In~\cite{berberich2022linear}, the rank condition~\eqref{eq:rank_condition_affine} is enforced via a data-dependent result established in~\cite[Theorem~1]{martinelli2022data}, referred to therein as the \textit{fundamental lemma for affine systems}. The result generalizes the fundamental lemma for \gls{LTI} systems~\cite{vanWaarde2020willems}, similarly requiring that the input $u_d$ used to collect the data be persistently exciting of a sufficiently high order. Specifically, under the assumption that the system is controllable and that the input $u_d$ is persistently exciting of order $n + L + 1$, the rank condition~\eqref{eq:rank_condition_affine} is guaranteed to hold.  For completeness, we recall the classical definition of \textit{persistence of excitation} and the relevant result next.

The sequence ${u \in (\mathbb{R}^m)^{\mathbf{T}}}$ is said to be \emph{persistently exciting of order ${L\in\mathbf{T}}$} if 
\beq 
\rank \, H_L(u) = mL. \label{eq:persistence_of_excitation_linear}
\eeq

\begin{theorem} ~\cite[Theorem~1]{martinelli2022data} \label{thm:martinelli22}
Consider the system~\eqref{eq:system_nonlinear_linearization_affine} and assume that the pair $(A, B)$ is controllable. Let 
\(u_d \in (\R^m)^{\mathbf{T}}\),
\(x_d \in (\R^n)^{\mathbf{T}}\),
\(y_d \in (\R^p)^{\mathbf{T}}\)  be an input/state/output trajectory of  system~\eqref{eq:system_nonlinear_linearization_affine}. Assume that the input $u_d$ is persistently exciting of order \(n + L+1\). 
Then the rank condition~\eqref{eq:rank_condition_affine} holds.
\end{theorem}

Theorem~\ref{thm:martinelli22} builds on assumptions analogous to those used for \gls{LTI} systems in~\cite{vanWaarde2020willems} to ensure that the finite-horizon behavior of an \gls{ATI} system is characterized by the affine image of a raw data matrix. However, as we will argue in this paper, the persistence of excitation condition required by Theorem~\ref{thm:martinelli22} is stronger than necessary and, in practice, it translates to needing longer experiments or richer input signals than required.  One aim of this paper is to provide sharper conditions that reduce data requirements. Another is to clarify what representations an affine systems admit and to what extent the resulting data matrices retain information about the system. 

\section{Affine systems} \label{sec:affine_systems}

Behavioral systems theory models systems as sets of trajectories~\cite{willems1986timea,willems1986timeb,willems1987timec}. Given a \textit{time set}~${\T}$ and a \textit{signal set}~$\W$, a 
\textit{trajectory} is any function ${w: \T \to \W}$ and
a \textit{system} is defined by its \textit{behavior} $\B$, \textit{i.e.}, a subset of the set of all possible trajectories $(\W)^{\T}$.  
Formally, a \textit{system} is defined as a triple ${\Sigma=(\T,\W,\B)}$ where $\T$ is the \textit{time set}, $\W$ is the \textit{signal set}, and $\B \subseteq (\W)^{\T}$ is the \textit{behavior} of the system. In this spirit, a \textit{model} of a system is given by its behavior, and a \textit{model class}  $\mathcal{M}$ is a family of subsets of the set of possible trajectories  $ (\W)^{\T}$.  Throughout this work, we exclusively focus on \textit{discrete-time} systems, with ${\T = \Z}$
and ${\W = \R^q}$,  and adapt definitions accordingly. Thus, we use the terms \textit{sequence}, \textit{time series}, and \textit{trajectory} interchangeably. By convention, we also identify systems with their behaviors and use  ``model,'' ``system,'' and ``behavior'' synonymously. 

A system is \textit{affine} if its behavior $\B$ is an affine set, \ie, 
$$ v,w \in\B \text{ and } \alpha\in\R ~ \Rightarrow ~ \alpha v+(1-\alpha)w\in\B. $$
A system $\B$ is \textit{linear} if its behavior $\B$ is a linear subspace, \ie, 
$$ 0 \in \B, \quad v,w \in\B \text{ and } \alpha,\beta\in\R ~ \Rightarrow  ~ \alpha v+\beta w\in\B. $$
Clearly, any linear system is affine. A system is \textit{time-invariant} if its behavior $\B$ is shift-invariant, \textit{i.e.}, ${\sigma \B = \B}$, and \textit{complete} if  ${w\in\B}$ if and only if ${w|_{[t_0,t_1]} \in \B|_{[t_0,t_1]}}$ for all ${t_0,t_1 \in \Z}$ such that ${t_0 < t_1}$. 
Completeness of a linear system is equivalent to its behavior being closed in the topology of pointwise convergence~\cite[Proposition 4]{willems1986timea}. 
The model class of all complete, \gls{LTI} and \gls{ATI} systems, with ${\T = \Z}$ and ${\W = \R^q}$, are denoted by $\L^q$ and $\A^q$, respectively. We also omit the superscript $q$ when its value is immaterial or unspecified. Identifying a system with its behavior, we often write $\B \in \L$ and $\B \in \A$.  Clearly, any \gls{LTI} system is \gls{ATI},  but not \textit{vice versa}, that is, $\mathcal{L} \subset \mathcal{A}$.

The model classes of \gls{LTI} and \gls{ATI} systems are closely related, suggesting that many properties of \gls{LTI} systems extend naturally to \gls{ATI} ones. However, we show that accounting for structural differences occasionally reveals  unexpected insights, such as weaker conditions in certain results. We show that \gls{ATI} systems inherit key similarities with \gls{LTI} systems — system representations, input–output partitions, controllability, and integer invariants — but also important differences, including the need to for conditions ensuring that certain representations do give rise to non-empty sets of trajectories and the opportunity to refine the classical persistence of excitation condition for \gls{ATI} systems to yield sharper, less conservative data requirements.

\subsection{Offset representations} 

An affine system $\mathcal{B}$ naturally has an associated linear system, which we call the \textit{difference system}, defined as 
\[\mathcal{B}_\textup{dif}:= \mathcal{B}-\mathcal{B}  \]
where $A - B$ denotes the elementwise subtraction of sets $A$ and $B$, defined as
$${A - B := \{ a - b \mid a \in A,\, b \in B \}.}$$
Analogously, the elementwise addition of sets $A$ and $B$ is defined as
$$
A + B := \{\, a+b \mid a \in A,\; b \in B \,\}.
$$
Any affine system can be viewed as a translation of a linear subspace, so it admits the following representation, presented previously as \cite[Theorem~1]{Markovsky2025Affine}.

\begin{theorem}[Offset representations] \label{thm:lin+offset} 
Consider a  non-empty  system $\mathcal{B} \in \mathcal{A}$. Then $\mathcal{B}_\textup{dif} \in \mathcal{L}$ and, for any $\wo\in\mathcal{B}$,
$$\mathcal{B}= \mathcal{B}_\textup{dif} + \{ \wo\} .$$ 
\end{theorem} 

\noindent
Theorem~\ref{thm:lin+offset} reveals a fundamental link between affine and linear systems. It enables the use of all system representations of linear systems, such as kernel representations, state-space realizations, and transfer functions, for the associated linear behavior $\mathcal{B}_\textup{dif}$. The affine system $\mathcal{B}$ may be defined by adding any trajectory ${\wo \in \mathcal{B}}$, which we refer to as an \textit{offset trajectory}. The case in which the offset trajectory $\wo$ can be chosen as a constant has been studied in the context of system identification~\cite{Markovsky2025Affine}. 
 
\subsection{Kernel representations} 

A classical result in behavioral system theory establishes that every complete \gls{LTI} system ${\B \in \L^q}$ admits a \textit{kernel representation}~\cite[Theorem 5]{willems1986timea}, that is, 
\beq \label{eq:kernel_representation_linear}
\B = \Ker \, R(\sigma) := \left\{w \,\mid \, R(\sigma)w = 0\right\} ,
\eeq
where  $R(\sigma)$ is the operator defined by the polynomial matrix 
$$R(\xi) = R_0 +R_1 \xi +\cdots+ R_{d}\xi^{d},$$ 
with $R_i\in\R^{g\times q}$  for any integer $i\in [0,d] $.  Without loss of generality, one may also assume  that $\Ker \, R(\sigma)$  is a \emph{minimal}  kernel representation of $\B$,  \textit{i.e.},  $R(\xi)$ has full row rank~\cite{willems1989models}.

Given a system $ {\mathcal{B} \in \mathcal{A}}$, our goal in this section is to extend classical kernel representations from the linear to the affine case. One naturally anticipates a close relationship between representations of the system $\mathcal{B}$ and those of its difference system $\mathcal{B}_{\mathrm{dif}}$, but how to make this relationship explicit is more subtle. A naive approach might attempt to construct representations of the system $\mathcal{B}$ by first representing its difference system $\mathcal{B}_\textup{dif}$, and then recovering a representation of the system $\mathcal{B}$ via the offset representation ${\mathcal{B} = \mathcal{B}_\textup{dif} + \{ \wo\}},$ with ${\wo \in \mathcal{B}}$, using Theorem~\ref{thm:lin+offset}.
This strategy, however, faces two fundamental obstacles. First, an affine system need not admit a finitely parametrized offset trajectory, such as a constant trajectory $\wo$ (see Example~\ref{ex:simple ex}). Second, even if such a trajectory exists, it is not clear how to determine it in practice.
 
\begin{example}[An \gls{ATI} system without a constant offset trajectory] \label{ex:simple ex}
    Consider the scalar \gls{ATI} system 
	\[\B = \{ w \mid \sigma w = w+1\}.\] 
Any trajectory of the behavior $\B$ is given as $\wo(t)=a+t$ for some $a\in\R$. Thus, although the defining equation ${\sigma w = w + 1}$ involves a constant offset, the behavior $\B$ has no constant trajectory.
\end{example}

\noindent
The following theorem clarifies the link between kernel representations of difference systems and their counterparts for \gls{ATI} systems. Moreover, it shows that the situation illustrated in Example~\ref{ex:simple ex} is general: \textit{every} \gls{ATI} system admits a representation defined by constant coefficients of the form
$$ \Ker_c \, R(\sigma) := \{ w \mid R(\sigma) w = c \} . $$

\begin{theorem}[Affine kernel representations]\label{thm:imp rep}
Consider a  non-empty  system $\mathcal{B} \in \mathcal{A}^q$ and a polynomial matrix $R\in\mathbb{R}^{g\times q}[\xi]$. Then
\begin{equation}\label{eq:affine_kernel_representation} 
\mathcal{B}   = \Ker_c \, R(\sigma)  
\end{equation}
for some $c\in\mathbb{R}^g$  if and only if 
\begin{equation}\label{eq:linear_kernel_representation}
\mathcal{B}_\textup{dif} = \Ker \, R(\sigma). 
\end{equation}
\end{theorem}
\noindent
Since any  non-empty  system in $\L^q$ admits a kernel representation of the form~\eqref{eq:linear_kernel_representation}, we can conclude that any system in $\A^q$ admits a representation of the form \eqref{eq:affine_kernel_representation}. We will refer to \eqref{eq:affine_kernel_representation} as a \textit{kernel representation} of the behavior $\B$, echoing the established terminology for~\eqref{eq:kernel_representation_linear}. The qualifiers \textit{affine} and \textit{linear} will be used, if necessary, to emphasize the assumptions on the underlying behavior. Theorem~\ref{thm:imp rep} in fact strengthens this observation: \emph{every} kernel representation of the difference behavior $\mathcal{B}_{\mathrm{dif}}$ induces an affine kernel representation of the behavior $\B$. One thing this allows us to do is to define the \textit{lag} of $\B\in\A$, denoted $\ell(\B)$ as the minimal degree of $R(\xi)$ over all affine kernel representations of $\B$.
 
It is important to highlight a key distinction between the linear and affine cases: while every \gls{LTI} kernel representation defines a nonempty behavior, an affine one may not. For \gls{LTI} systems, the correspondence between behaviors and kernels of polynomial matrices is well understood: every $\B \in \L^q$ admits a kernel representation of the form~\eqref{eq:kernel_representation_linear}, and conversely every polynomial matrix $R(\xi)$ defines a linear system $\Ker \, R(\sigma) \in \L^q$. Moreover, two polynomial matrices $R(\xi)$ and $R'(\xi)$ define the same system if they differ by the addition or removal of zero rows and left-multiplication with a unimodular matrix (cf. \cite[Corollary 3.6.3]{willems1997introduction}).
Equivalently, the map  ${R(\xi) \mapsto  \Ker  R(\sigma)}$
defines a one-to-one correspondence between equivalence classes of polynomial matrices
and elements of $\L^q$.  

In contrast to the linear case, the correspondence does not carry over directly for affine systems, as the set \( \Ker_c R(\sigma) \) may be empty.  A pair \( (R(\xi), c) \) is said to be \textit{consistent} if it defines a non-empty affine system. A detailed analysis of this consistency property is deferred to Appendices~\ref{sec:equivalence} and~\ref{sec:consistency}.
In particular, Appendix~\ref{sec:equivalence} extends the classical equivalence classes to the affine case and characterizes when two affine kernel representations define the same system, while Appendix~\ref{sec:consistency} provides testable conditions for consistency.  
In general, the map ${(R(\xi), c) \mapsto \Ker_c R(\sigma)}$ induces a bijection between the set of all nonempty affine systems and the equivalence classes (cf.~Theorem~\ref{thm:equiv}) of \textit{consistent} pairs \( (R(\xi), c) \).

\subsection{Input-output representations}

This section introduces the affine counterpart of the classical \emph{input--output representations} for \gls{LTI} systems~\cite{willems1986timea}. To this end, we build on standard notions from behavioral systems theory, such as input--output partitions, which are sufficiently general to apply to a broad range of model classes, including both \gls{LTI} and \gls{ATI} systems. 
Since these notions may not be familiar to all readers, we provide a brief overview and refer to~\cite{willems1986timea} for a more detailed exposition.

Consider a behavior $\B$ and partition its variables as $w = \left[\begin{smallmatrix} u \\ y\end{smallmatrix}\right]$, with $u(t)\in\mathbb{R}^{m}$ and $ y(t) \in \mathbb{R}^{p}$. Any such partition induces natural projections ${ \pi_u:  w \mapsto u}$ and ${ \pi_y:  w \mapsto y}$. The variable $u$ is called \textit{free} for $\B$ if for every $u$ there exists $y$ such that $\left[\begin{smallmatrix} u \\ y\end{smallmatrix}\right]\in\B$. In other words, $u$ is free if ${\pi_u\B = (\mathbb{R}^{m})^\Z}$; it is said to be \textit{maximally free} if none of the components of $y$ can be chosen freely. If $u$ is maximally free, we call $u$ an \textit{input} and the partition an \textit{input-output partition}. Note that the dimension of all choices of input, the \textit{input cardinality} is the same, and we denote it by $\mathbf{m}(\B)$. Similarly, the \textit{output cardinality} is denoted $\mathbf{p}(\B)= q- \mathbf{m}(\B)$. 

Without loss of generality, the variables of any complete \gls{LTI} system can be \textit{permuted} to an input-output partition.  Given ${\B \in \L^q}$ with input-output partition 
$w = \left[\begin{smallmatrix} u \\ y\end{smallmatrix}\right]$,
with ${u \in \R^m}$ and ${y \in \R^p}$, an input-output representation of $\mathcal{B}$ is defined as
\[
\mathcal{B} := \{\, w \mid P(\sigma) y = Q(\sigma) u \,\},
\]
with $P(\xi) \in \R^{g \times p}$ and $Q(\xi) \in \R^{g \times m}$ polynomial matrices.  Moreover, $P$ and $Q$ can be chosen to be \emph{left–coprime}, in which case the representation is unique up to left-multiplication by a unimodular matrix~\cite{willems1986timea}.  

Our  next   result shows that input-output partitions for \gls{ATI} systems correspond precisely to those of their difference behaviors, and  \textit{vice versa}.

\begin{theorem}[Input-output partitions and difference systems] \label{thm:io} 
Consider a  non-empty  system $\mathcal{B} \in \mathcal{A}$.
Then $w = \left[\begin{smallmatrix} u \\ y\end{smallmatrix}\right]$ is an input-output partition of $\B$ if and only if it is an input-output partition for $\B_\textup{dif}$. \end{theorem}

\noindent
Theorem~\ref{thm:io} allows us to prove a result analogous to Theorem~\ref{thm:imp rep} for input–output representations; the proof is similar and, hence, omitted for brevity.    

\begin{corollary}[Affine input-output representations] \label{thm:affine_IO_representations}
Consider a  non-empty  system ${\mathcal{B} \in \mathcal{A}^q}$ and a polynomial matrix ${R\in\mathbb{R}^{g\times q}[\xi]}$.  Let $w = \left[\begin{smallmatrix} u \\ y\end{smallmatrix}\right]$ be an input-output partition, where ${u\in\mathbb{R}^m}$ is an input and ${y\in\mathbb{R}^p}$ is an output, with ${p=q-m}$. Let $P\in\mathbb{R}^{p\times p}[\xi]$ and $Q\in\mathbb{R}^{p\times m}[\xi]$. Then  
        \begin{equation}\label{eq:affine_io_representation} \mathcal{B} = \{ w \mid P(\sigma) y = Q(\sigma)u + c\} \end{equation}	
 	for some $c\in\mathbb{R}^p$ if and only if 
        \begin{equation}\label{eq:linear_io_representation} \mathcal{B}_\textup{dif} = \{ w \mid P(\sigma) y = Q(\sigma)u \} .\end{equation}
\end{corollary}

\subsection{State-space representations}

The key feature of a state variable is its \emph{Markovian} property: the current state summarizes all relevant past information so that future trajectories depend only on the present state and the current and future inputs. 
For \gls{LTI} systems, this property yields the familiar state–space representation  
\beq\label{eq:state-space-linear} 
 \sigma x = Ax + Bu, \quad y = Cx + Du.
\eeq
 For \gls{ATI} systems, however, the appropriate form of a state–space representation is less immediate.   One may conjecture that constant offsets in the state–update and output equations should be sufficient, leading to affine state–space equations of the form  
\begin{equation} \label{eq:state-space-affine} 
\sigma x = Ax + Bu + E, \quad    y   = Cx + Du + F,
\end{equation}
with $x(t) \in \R^n$, $u(t) \in \R^m$, and $y(t) \in \R^p$.  
A key question is whether \emph{every} system $\B \in \A$ can indeed be represented in this way.  
Our next result shows that the answer is affirmative. Furthermore, \emph{any} realization of the corresponding difference behavior $\B_{\mathrm{dif}}$ gives rise to an affine state–space representation of the associated affine system of the form~\eqref{eq:state-space-affine}. 

\begin{theorem}[Affine state-space representations]\label{thm:affine states vs linear states}
Consider a  non-empty  system $\mathcal{B} \in \mathcal{A}^q$. Let $w = \left[\begin{smallmatrix} u \\ y\end{smallmatrix}\right]$ be an input-output partition, where ${u\in\mathbb{R}^m}$ is an input and ${y\in\mathbb{R}^p}$ is an output, with ${p=q-m}$. Let ${n\in\mathbb{N}}$ and let $A\in\mathbb{R}^{n\times n}$, $B\in\mathbb{R}^{n\times m}$, $C\in\mathbb{R}^{p\times n}$, and $D\in\mathbb{R}^{p\times m}$. Then, 
\begin{equation}\label{eq:affine_state_space_representation}   
\mathcal{B} =  
\big\{\, w = \bsm u \\ y \esm \ | \ \text{there is $x$, such that \eqref{eq:state-space-affine}} \,\big\}, 
\end{equation}
for some $E\in\mathbb{R}^n$ and $F\in\mathbb{R}^p$ if and only if 
\begin{equation}\label{eq:linear state}  
\B_\textup{dif} =  
\big\{\, w = \bsm u \\ y \esm \ | \ \text{there is $x$, such that  \eqref{eq:state-space-linear}} \,\big\}. 
\end{equation}
\end{theorem}

Theorem~\ref{thm:affine states vs linear states} has two main implications. First, it indirectly defines a well-defined notion of \emph{order} for any \gls{ATI} system, showing that any $\B\in\A^q$ admits a \textit{minimal} affine state-space representation~\eqref{eq:state-space-affine}, where ${n \in \N}$ is as small as possible. The smallest such $n$ is called the \textit{order} or \textit{state cardinality}, and denoted $\mathbf{n}(\B)$. Second, it relates the \emph{internal} properties of an \gls{ATI} system, such as controllability, observability, and related notions, directly to those of its associated difference system. In Section~\ref{sec:props}, we make use of both implications to characterize the controllability of \gls{ATI} systems and associated invariants. Finally, note that one cannot, in general, restrict to the special cases $E=0$ or $F=0$: while both yield affine systems, they do not cover the full model class $\A$.

\section{System properties} \label{sec:props}

We now turn to two fundamental system-theoretic aspects — controllability and integer invariants — that play a central role in developing a fundamental lemma for affine systems.

\subsection{Controllability} \label{sssec:controllability}

Controllability is a key concept in systems and control, capturing the ability to steer a system between trajectories over finite time. For our purposes, it is essential to establish a basic result for \gls{ATI} systems. Given the structural link between an \gls{ATI} system $\mathcal{B}$ and its associated difference system $\mathcal{B}_\mathrm{dif}$, it is natural to ask: can the controllability of one be inferred from the other? We show that this is indeed the case.

In behavioral system theory, a system $\mathcal{B}$ is said to be \emph{controllable} if, for any trajectories $w_1, w_2 \in \mathcal{B}$ and any  $t_1 \in \Z$,   there exists $t_2 \ge t_1$ and a trajectory ${w} \in \mathcal{B}$ such that
\begin{equation} \label{eq:cont patch}
{w}(t) = 
\begin{cases}
w_1(t), & t < t_1, \\
w_2(t), & t \ge t_2.
\end{cases}
\end{equation}
In other words, the system is controllable if any two trajectories in $\mathcal{B}$ can be patched in finite time. We now make an elementary, but useful observation.

\begin{lemma}[Controllability is translation-invariant] \label{lemma:controllability_is_translation-invariant}
Consider a  non-empty  system $\mathcal{B} \subseteq (\mathbb{R}^q)^{\mathbb{Z}}$ and let $\wo \in (\mathbb{R}^q)^{\mathbb{Z}}$. Then $\mathcal{B}$ is controllable if and only if $\mathcal{B} + \{ \wo \}$ is controllable.
\end{lemma}

\noindent
Lemma~\ref{lemma:controllability_is_translation-invariant} holds without further assumptions. As a consequence, we can use Theorem~\ref{thm:lin+offset} to conclude that any \gls{ATI} system $\mathcal{B}$ is controllable if and only if its difference system $\mathcal{B}_\textup{dif}$ is controllable.  As a consequence, standard controllability tests can be applied directly to the difference system $\mathcal{B}_\textup{dif}$. We now formalize this observation in the following statement.

\begin{corollary}[Controllability of affine systems] \label{cor:controllability-of-affine-systems}
Consider a  non-empty  affine system $\mathcal{B}$ with associated difference system $\mathcal{B}_\mathrm{dif}$. 
Then $\mathcal{B}$ is controllable if and only if $\mathcal{B}_\mathrm{dif}$ is controllable. 
If, in addition, $\mathcal{B}$ is time-invariant, then the following hold:
\begin{itemize}
    \item[1.] If $\mathcal{B}$ admits the kernel representation~\eqref{eq:affine_kernel_representation}, then $\mathcal{B}$ is controllable if and only if 
    \[
        \rank \, R(\lambda) \;\; \text{is constant for all } \lambda \in \mathbb{C} .
    \]
    \item[2.] If $\mathcal{B}$ admits the minimal state space representation~\eqref{eq:affine_state_space_representation}, then $\mathcal{B}$ is controllable if and only if 
    \[
        \rank \begin{bmatrix} B & AB & \cdots & A^{n-1}B \end{bmatrix} = n .
    \]
\end{itemize}
\end{corollary}

\subsection{Integer invariants}\label{sssec:integer_invariants}
A fundamental insight of behavioral system theory is that every \gls{LTI} system ${\B \in \L^q}$ is characterized by a collection of \emph{integer invariants}~\cite{willems1986timea}. These invariants are intrinsic: they reflect the structure of the behavior itself and do not depend on the choice of representation. For any affine system $\B\in\A^q$, we have  indirectly  shown that the \emph{variable cardinality} $\mathbf{q}(\B)$, the \emph{input cardinality} $\mathbf{m}(\B)$, the \emph{output cardinality} $\mathbf{p}(\B)$, the \emph{state cardinality} $\mathbf{n}(\B)$, and the  \emph{lag} $\boldsymbol{\ell}(\B)$ are also  well-defined,  as the integer invariants coincide with those of the corresponding difference system $\B_{\textup{dif}}$. 
However, the definitions of input cardinality, state cardinality, and lag are purely existential and therefore do not immediately provide a constructive way to determine them. To bridge this gap, we derive the following result, which generalizes the linear case.

\begin{lemma}[Integer invariants] \label{lemma:integer-invariants}
Consider a non-empty system $\B \in \A^q$.  
Define the sequences
\[
\begin{alignedat}{2}
d_t &:= \dim \B|_{[1,t]}, &\quad &t \in \N,\\
\rho_t &:= d_t - d_{t-1}, &\quad &d_0 := 0,\\
\gamma_t &:= \rho_{t-1} - \rho_t, &\quad &\rho_0 := q,
\end{alignedat}
\]
where $\dim$ denotes the dimension of an affine set.  
Then
\[
\begin{aligned}
\mathbf{m}(\B) &= \lim_{t \to \infty} \rho_t,\\[2pt]
\mathbf{n}(\B) &= \sum_{t=1}^{\infty} t\,\gamma_t,\\[2pt]
\boldsymbol{\ell}(\B) &= \min\{\, t \in \N \mid \rho_t = \mathbf{m}(\B)\,\}.
\end{aligned}
\]
\end{lemma}

\section{Fundamental lemma for affine systems}  \label{sec:fundamental_lemma_ATI}

Input design concerns selecting inputs that sufficiently excite system dynamics to enable identifiability and effective control. For \gls{LTI} systems, this idea is formalized by the \emph{fundamental lemma}~\cite{willems2005note,vanWaarde2020willems}, which asserts that all trajectories of a controllable \gls{LTI} system can be constructed from a single sufficiently exciting trajectory.  The condition that guarantees this property is the classical persistence of excitation~\eqref{eq:persistence_of_excitation_linear}, which has been instrumental in shaping the theory of input design~\cite{ljung1999system}.  Extensions to affine systems have typically relied on the same condition, see, e.g.,~\cite{martinelli2022data,berberich2022linear}.

However, we argue that this classical condition is overly conservative in the affine setting. The structure of affine systems calls for a dedicated notion of persistence of excitation, one that better reflects the geometry of the model class. In what follows, we introduce a new persistence of excitation condition tailored to \gls{ATI} systems and show that it enables a generalization of the fundamental lemma~\cite{vanWaarde2020willems} beyond the \gls{LTI} case. This perspective emphasizes that the central issue is not merely verifying a generalized excitation condition, but rather understanding how input design guarantees the constructive power of the fundamental lemma.  
As a result, we obtain provably lower data requirements compared to the state of the art (cf.,~\cite{martinelli2022data,berberich2022linear}), highlighting the importance of adapting excitation requirements to the model class under study.
 
A sequence $u \in (\mathbb{R}^m)^{\mathbf{T}}$ is said to be \emph{persistently exciting of order $L\in\mathbf{T}$ for the model class $\A$} if 
\beq 
\rank \,
	\bma 
		\begin{array}{c}  
		H_L(u) \\ 
		\mathbb{1}^{\transpose} 
		\end{array} 
	\ema  
=  mL+1.\label{eq:persistence_of_excitation_affine}
\eeq
To distinguish this notion from the classical persistence of excitation condition~\eqref{eq:persistence_of_excitation_linear}, a sequence $u \in (\mathbb{R}^m)^{\mathbf{T}}$ is said to be  \emph{persistently exciting of order $L\in\mathbf{T}$ for the model class $\mathcal{L}$} if the rank condition~\eqref{eq:persistence_of_excitation_linear} holds.

\begin{remark}[Geometry of persistence of excitation]
From a geometric perspective, the classical and affine notions of persistence of excitation differ in how finite windows of the input sequence are represented. For linear systems, the persistence of excitation condition~\eqref{eq:persistence_of_excitation_linear} for requires that these windows can be expressed as \emph{linear} combinations of the columns of~$H_L(u)$.  
For affine systems, the persistence of excitation condition~\eqref{eq:persistence_of_excitation_affine} augments~$H_L(u)$ with a row of ones, enforcing that the windows can instead be written as \emph{affine} combinations of its columns.  
The additional dimension in the rank condition (arising from the row of ones) is a reflection of the extra degree of freedom that separates affine from linear independence.  
We return to this point in Section~\ref{sec:discussion}.
\end{remark}

With a suitable notion of persistence of excitation in place, we are now ready to establish  a fundamental lemma for the model class of \gls{ATI} systems.   We begin by presenting a behavioral formulation akin to \cite[Theorem~1]{willems2005note}.

\begin{theorem}[Fundamental lemma for \gls{ATI} systems] \label{thm:fundamental_lemma_ATI_systems}
Consider a controllable  non-empty system  $\mathcal{B}\in\mathcal{A}^q$. Let $w = \left[\begin{smallmatrix} u \\ y\end{smallmatrix}\right]$ be an input-output partition of $\B$ and let $w_d = \left[\begin{smallmatrix} u_d \\ y_d\end{smallmatrix}\right]\in(\mathbb{R}^{q})^\mathbf{T}$ be a trajectory of $\B$. Assume that $u_d$ is persistently exciting of order $\mathbf{n}(\B)+L$ for the model class $\A$. Then every trajectory $w$ of length \(L\) of the system $\B$ can be expressed as
    \[
    w = {H}_L(w_{d}) g,
    \]
    for some \(g \in \R^{T-L+1}\) such that $\mathbb{1}^{\transpose} g = 1$. 
\end{theorem}

Theorem~\ref{thm:fundamental_lemma_ATI_systems} establishes that every finite trajectory of an \gls{ATI} system can be constructed from a single sufficiently exciting trajectory, thereby extending the fundamental lemma beyond the linear case. 
For a system described by a minimal state-space representation~\eqref{eq:state-space-affine}, we provide a reformulation in line with \cite[Theorem~1]{vanWaarde2020willems},  highlighting that our extension aligns with established results for \gls{LTI} systems.  

\begin{theorem}[Fundamental lemma for state-space \gls{ATI} systems] \label{thm:fundamental_lemma_state-space_ATI_systems}
Consider the system~\eqref{eq:state-space-affine} and assume that the pair $(A, B)$ is controllable. Let 
\(u_d \in (\R^m)^{\mathbf{T}}\),
\(x_d \in (\R^n)^{\mathbf{T}}\),
\(y_d \in (\R^p)^{\mathbf{T}}\)  be an input/state/output trajectory of system~\eqref{eq:state-space-affine}. Assume that the input $u_d$ is persistently exciting of order \(n + L\) for the model class $\A$. Then the following statements hold.
\begin{itemize}
    \item[(i)] The rank condition
    \beq  \label{eq:rank_condition_aff}
    \rank \, 
    \scalebox{1}{$
    \bma  
    \begin{array}{c}
    H_1(x_d) \\
    H_L(u_d) \\
    \mathbb{1}^{\transpose}
    \end{array}
    \ema = mL + n + 1$}
    \eeq
    holds.
    \item[(ii)] Every trajectory $\left[\begin{smallmatrix} u \\ y\end{smallmatrix}\right] \in \R^{(m+p)\mathbf{L}}$ of
    system~\eqref{eq:state-space-affine} of length \(L\) can be expressed as
    \[
    \begin{bmatrix}
    u \\
    y
    \end{bmatrix}
    =
    \begin{bmatrix}
    {H}_L(u_{d}) \\
    {H}_L(y_{d})
    \end{bmatrix} g,
    \]
    for some \(g \in \R^{T-L+1}\) such that $\mathbb{1}^{\transpose} g = 1$.
\end{itemize}
\end{theorem}

\noindent
Theorems~\ref{thm:fundamental_lemma_ATI_systems} and~\ref{thm:fundamental_lemma_state-space_ATI_systems} inherently leverage the rank condition
\begin{equation}\label{eq:ATI-GPE-first}
\rank \begin{bmatrix}
    {H}_L(w_{d}) \\
    \mathbb{1}^{\transpose}
    \end{bmatrix} = d_L+1,
\end{equation}
where $d_L$ is defined as in Lemma~\ref{lemma:integer-invariants}. 
By Lemma~\ref{lemma:integer-invariants}, if $L\geq\boldsymbol{\ell}(\B)$, the condition~\eqref{eq:ATI-GPE-first} reduces to 
\begin{equation}\label{eq:ATI-GPE}
    \rank \begin{bmatrix}
    {H}_L(w_{d}) \\
    \mathbb{1}^{\transpose}
    \end{bmatrix} = \mathbf{m}(\B)L + \mathbf{n}(\B)+1.
\end{equation}
This condition is referred to as the \emph{generalized affine persistency of excitation} (GAPE) condition. Under this condition, we can use elementary linear algebra to find an integer $g\in\N$, matrices $R_0,\ldots,R_L\in\mathbb{R}^{g\times q}$, and a vector $c\in\mathbb{R}^g$ such that, for any trajectory $w$ of length \(L\) of the system $\B\in\A^q$, we have 
\[ \begin{bmatrix} R_0 & \cdots & R_L\end{bmatrix} {w} =c  
\]
if and only if
\[ 
{w} = {H}_L(w_{d}) g, \text{ and }\mathbb{1}^{\transpose} g = 1.  
\]
This, in turn, allows one to define a polynomial matrix 
\[R(\xi):= R_0+\cdots+R_L\xi^L\] 
and conclude that 
\[ \B = \Ker_c R(\sigma ). \]
In other words, the behavior of a controllable \gls{ATI} system can be inferred from a single experiment performed using an input that is persistently exciting of order greater than or equal to $\mathbf{n}(\B)+ \boldsymbol{\ell}(\B)$ for the model class $\A$. 
 The same construction can be extended to accommodate multiple experiments, in the spirit of \cite{vanWaarde2020willems}, where trajectories from different experiments are used to enrich the data matrix. 

\begin{remark}[Sufficiency of the excitation conditions] \label{rem:sufficiency-of-the-excitation-conditions}
The persistence of excitation conditions stated in Theorems~\ref{thm:fundamental_lemma_ATI_systems} and~\ref{thm:fundamental_lemma_state-space_ATI_systems} are sufficient but not necessary, consistent with the linear theory. In practice, input signals of lower excitation order may still generate trajectories that capture the system’s behavior, yielding informative data matrices and substantially reduced data requirements. Section~\ref{sec:example} illustrates this point with a numerical example.
\end{remark}

\section{Discussion} \label{sec:discussion}

This section examines the implications of our results for affine systems and the associated notion of persistence of excitation. First, we establish that the class of persistently exciting inputs for affine systems is strictly larger than that for linear systems: the affine and linear conditions are not equivalent. This distinction translates directly into reduced data requirements for identification and control. Then, we quantify this reduction by comparing the minimal sequence lengths needed in each case and highlight the gap. By revisiting recent literature, particularly the fundamental lemma for affine systems introduced in~\cite{martinelli2022data}, we show that their overly conservative conditions stem from an implicit lifting of the affine system to a higher-order linear one. This viewpoint aligns with observations in~\cite{Markovsky2025Affine}, and naturally explains the appearance of  ``lifted'' trajectories of the form $\left[\begin{smallmatrix} w \\ 1\end{smallmatrix}\right]$.  Finally, we offer a geometric interpretation of this lifting operation. Affine behaviors may be viewed as ``slices'' of linear behaviors in projective space, where additive offsets correspond to homogenization in the extended variable. This projective perspective not only clarifies the algebraic structure of affine systems, but also motivates \textit{a fortiori} the rank conditions underlying our main results.

\subsection{Distinct model classes, distinct excitation conditions}

Affine and linear systems admit different structural descriptions. It is therefore natural that the corresponding persistence of excitation conditions also differ. In particular, persistence of excitation conditions for the affine model class $\A$ are \textit{not} equivalent to those for the linear model class $\L$: one does not automatically imply the other.   The following lemma clarifies some aspects of their relationship.

\begin{lemma}[Affine  and linear persistence of excitation] \label{lemma:PE-relationship}
Let $u \in (\R^m)^{\mathbf{T}}$ and $L \in \mathbf{T}$. If $u$ is persistently exciting of order $L$ for $\A$, then it is persistently exciting of order $L$ for $\L$. Conversely, if $u$ is persistently exciting of order $L+1$ for $\L$, then it is persistently exciting of order $L$ for $\A$.
\end{lemma}

\noindent
Next, we present an example showing that a sequence can be persistently exciting of a given order for the model class $\L$, but not for the model class $\A$.

\begin{example}[A\,persistently\,exciting\,sequence\,for\,$\L$\,but\,not\,for\,$\A$]
    Consider the scalar sequence 
    $$ u = (1, 2, 1, 2, 1, 2, \ldots ). $$
    The sequence is persistently exciting of any order $L\leq 3$ for the model class $\L$. However, it is only persistently exciting of order $L\leq 2$ for the model class $\A$.
\end{example}

\noindent
For the model class~$\A$, persistence of excitation is determined by a rank condition on a Hankel matrix with an additional row of ones, reflecting the constant component of the signal. Viewed in the frequency domain, this term captures the zero-frequency harmonic, whereas the remaining rows of the Hankel matrix correspond to nonzero-frequency components. Accordingly, if the input~$u$ contains no zero-frequency (constant) component, its order of persistence of excitation for the model class~$\A$ coincides with that for the model class~$\L$. Consequently, all familiar classes of persistently exciting signals for the model class~$\L$ extend naturally to the model class~$\L$  once the zero-frequency contribution is removed.
 
\subsection{Data requirements of different excitation conditions}

We now compare the minimal input lengths required to guarantee persistence of excitation for the model classes $\L$ and $\A$, respectively. A necessary condition for ${u \in (\mathbb{R}^{m})^\mathbf{T}}$ to be persistently exciting of order ${L\in \mathbf{T}}$ for the model class $\mathcal{L}$ is
\[
T - L + 1 \geq mL.
\]
 
\noindent 
By defining   $T_{L}(\mathcal{L})$ to be the minimal sequence length required to satisfy the persistence of excitation  condition,   we require
\[
T \geq T_{L}(\mathcal{L}) := (m+1)L - 1.
\]
Similarly, a necessary condition for ${u \in (\mathbb{R}^{m})^\mathbf{T}}$ to be persistently exciting of order ${L\in \mathbf{T}}$ for the model class $\mathcal{A}$ is
\[
T \geq T_{L}(\mathcal{A}) := (m+1)L  -1    .
\]
Theorem~\ref{thm:martinelli22} characterizes all trajectories of length \(L\) arising from a \gls{ATI} system as affine combinations
of the columns of a data matrix. However, this result relies on the input being persistently exciting of order \(n + L + 1\) for the model class $\L$, which imposes   the minimum data sequence length 
\[
T_{n + L + 1}(\mathcal{L}) = (m + 1)(n + L + 1) - 1.
\]
By contrast, Theorem~\ref{thm:fundamental_lemma_ATI_systems} establishes that all trajectories of length \(L\) arising from a \gls{ATI} system as affine combinations of the columns of the same data matrix using input sequence that are persistently exciting of order \(n + L\) for the model class $\A$, which imposes  the minimum data sequence length 
\[
T_{n+L}(\mathcal{A}) = (m+1)(n+L) -1 .
\]
The resulting sample complexity reduction is
\[
T_{n+L+1}(\mathcal{L}) - T_{n+L}(\mathcal{A}) = m+1.
\]
The sample complexity reduction scales linearly with the number of inputs~$m$, and can therefore become significant in multi-input systems. Furthermore, as discussed in Remark~\ref{rem:sufficiency-of-the-excitation-conditions}, the persistence of excitation conditions used in Theorems~\ref{thm:fundamental_lemma_ATI_systems} and~\ref{thm:fundamental_lemma_state-space_ATI_systems} are sufficient, but not necessary. Section~\ref{sec:example} illustrates, through a numerical example, that input sequences of lower excitation order may still generate trajectories that allow one to reproduce the behavior of a system, yielding informative data matrices with fewer samples.

\subsection{The projective geometry of affine systems}

Affine geometry generalizes linear geometry by preserving directional relationships between points and discarding the need for a fixed origin~\cite{eisenbud1995commutative}.  Affine spaces are precisely the sets preserved under \textit{affine transformations}, that is, maps that preserve affine combinations, defined by weighted sums whose coefficients sum to the multiplicative identity of the underlying field. In the Euclidean space \(\mathbb{R}^q\), such transformations take the form \(f(w) = Lw + \wo\), with \(L \in \mathbb{R}^{q \times q}\) and \(\wo \in \mathbb{R}^q\). These maps preserve collinearity and ratios of oriented line segments, but do not necessarily preserve distances or angles.

A classical way to study affine sets is to lift them to linear subspaces via \emph{homogenization}~\cite{eisenbud1995commutative}. For example, in the Euclidean space \(\mathbb{R}^q\), the map
\beq \label{eq:map_homogeneization}
w \mapsto  \begin{bmatrix} w \\ 1 \end{bmatrix}
\eeq
embeds \(\mathbb{R}^q\) into \(\mathbb{R}^{q+1}\), mapping affine sets of \(\mathbb{R}^q\) into linear subspaces of \(\mathbb{R}^{q+1}\). This construction is central in \textit{projective} geometry~\cite{eisenbud1995commutative}, where projective space \(\mathbb{P}^q\) is defined as the set of all lines through the origin in \(\mathbb{R}^{q+1}\). Points in \(\mathbb{R}^q\) are identified with equivalence classes of nonzero vectors in \(\mathbb{R}^{q+1}\) under scalar multiplication, defined as 
\[
[w] = \{ \lambda w : \lambda \in \mathbb{R} \setminus \{0\},\; w \in \mathbb{R}^{q+1} \setminus \{0\} \}.
\]
An affine space is then realized as the chart of \(\mathbb{P}^q\) corresponding to the slice defined by \(w_{q+1} \neq 0\), with the identification
\[
[w_1, \dots, w_q, w_{q+1}] \mapsto \left(\frac{w_1}{w_{q+1}}, \dots, \frac{w_q}{w_{q+1}}\right).
\]
In this chart, affine subspaces correspond to projective subspaces not intersecting the hyperplane at infinity
$$ \mathsf{H} = \{ w \in \mathbb{R}^{q+1} : w_{q+1} = 0 \}.$$ 
The resulting homogeneous representation describes affine relations as linear relations in projective coordinates. The homogenization map is injective and its image intersects each equivalence class \([w] \in \mathbb{P}^q\) precisely once in the chart defined by \(w_{q+1} = 1\), as illustrated in Figure~\ref{fig:homogenization}.

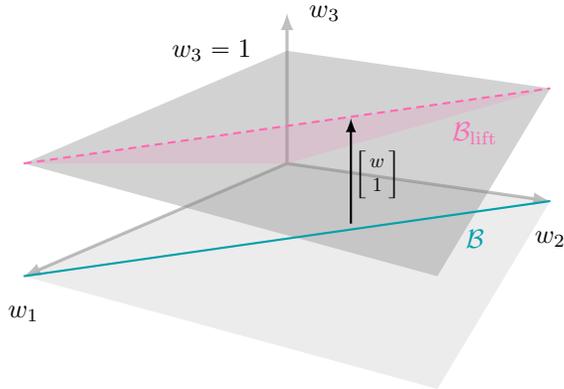
\begin{figure}[ht]
\centering
\begin{tikzpicture}

\draw[-latex, very thick, gray!50] (-2.5,4.5) -- (-2.5,6.5);
\draw[-latex, very thick, gray!50] (-2.5,4.5) -- (1,4);
\draw[-latex, very thick, gray!50] (-2.5,4.5) -- (-6,3);

\fill[gray, opacity=0.15] (-6,3) -- (-0.5,1.5) -- (1,4) -- (-2.5,4.5) -- cycle;
\fill[gray, opacity=0.35] (-6,4.5) -- (-0.5,3) -- (1,5.5) -- (-2.5,6) -- cycle;

\draw[thick, coolTeal] (-6,3) -- (1,4);
\draw[thick, densely dashed, hotPink] (-6,4.5) -- (1,5.5);

\fill[hotPink, opacity=0.15] (1,5.5) -- (-2.5,4.5) -- (-6,4.5) -- cycle;

\node[text=coolTeal] at (0,3.5) {$\mathcal{B}$};

\node[text=black] at (-3.5,6) {$w_3=1$};

\node[text=hotPink] at (0,4.9) {$\mathcal{B}_{\textup{lift}}$};
\node[black] at (-6,2.5) {$w_1$};
\node[black] at (1,3.5) {$w_2$};
\node[black] at (-2,6.5) {$w_3$};

\draw[-latex,  thick, black] (-1.65,3.7) -- (-1.65,5.1);
\node[black] at (-1.3,4.35) {\scalebox{0.75}{$\left[\!\!\begin{array}{c} w \\ 1 \end{array}\!\!\right]$}};
\end{tikzpicture}
\caption{An affine subspace \(\textcolor{coolTeal}{\mathcal{B}} \subseteq \mathbb{R}^2\) (solid) is embedded into a linear subspace \(\textcolor{hotPink}{\mathcal{B}_{\textup{lift}}} \subseteq \mathbb{R}^3\) (shaded) through the map~\eqref{eq:map_homogeneization}, whose image intersects each equivalence class \([w] \in \mathbb{P}^2\) exactly once in the chart defined by \(w_{3} = 1\) (dashed).}
\label{fig:homogenization}
\end{figure}

In behavioral systems theory, this construction provides a systematic tool for analyzing behaviors that define affine trajectory spaces. Once lifted via homogenization, such spaces admit linear representations and can be analyzed using the full machinery of linear system theory. For instance, given a sequence
$u \in (\mathbb{R}^m)^{\mathbf{T}}$ and the corresponding Hankel matrix 
\({H}_L(w_{d}) \in \mathbb{R}^{q \times T}\), with $L\in\mathbf{T}$,  the matrix
\[
\begin{bmatrix} {H}_L(u) \\ \mathbb{1}^\top \end{bmatrix} \in \mathbb{R}^{(mL+1) \times (T-L+1)}
\]
can be used to verify rank conditions and assess persistence of excitation—analogous to the linear case.

A further expression of this idea is found in the role played by the homogenization map~\eqref{eq:map_homogeneization} in representing \gls{ATI} systems. In Theorems~\ref{thm:imp rep} and \ref{thm:affine states vs linear states} as well as in Corollary~\ref{thm:affine_IO_representations}, this map enables a one-to-one correspondence between affine representations and those of the associated lifted linear systems. Theorem~\ref{thm:imp rep}, viewed through this lens, can be reformulated as follows.

\begin{lemma}[Homogeneous affine representations]
For every system \(\mathcal{B} \in \mathcal{A}^q\), there exist \(g \in \N \) and a polynomial matrix \(P(\xi) \in \mathbb{R}^{g \times (q+1)}[\xi]\) such that
\[
\mathcal{B} = \left\{ w \;\middle|\; P(\sigma) \begin{bmatrix} w \\ 1 \end{bmatrix} = 0 \right\}.
\]
\end{lemma}

\noindent
A parallel construction applies to state-space representations, as  preliminarily  discussed in~\cite{Markovsky2025Affine}. Given an \gls{ATI} system \(\mathcal{B} \in \mathcal{A}^q\) with minimal state space representation 
\[
\mathcal{B} = \left\{ w = \begin{bmatrix} u \\ y \end{bmatrix} \;\middle|\; \exists x \text{ s.t. } \begin{aligned}
\sigma x &= Ax + Bu + E, \\
y &= Cx + Du + F
\end{aligned} \right\},
\]
define the lifted \gls{LTI} system \(\mathcal{B}_{\mathrm{lift}} \in \mathcal{L}^q\) as
\[
\sigma \chi = \begin{bmatrix} A & E \\ 0_{1 \times n} & 1 \end{bmatrix} \chi + \begin{bmatrix} B \\ 0 \end{bmatrix} u, 
\quad 
y = [C \;\; F] \chi + D u.
\]
The lifted system \(\mathcal{B}_{\mathrm{lift}}\) is uniquely determined by \(\mathcal{B}\) and is independent of the particular state-space realization. Its relevance lies in the fact that
\[
w = \begin{bmatrix} u \\ y \end{bmatrix} \in \mathcal{B} 
\]
if and only if
\[
w \in \mathcal{B}_{\mathrm{lift}} \text{ and } \chi(0) = \begin{bmatrix} x(0) \\ 1 \end{bmatrix}.
\]
This equivalence allows standard \gls{LTI} techniques to be applied to simulation, analysis, and control of affine systems. The lifted system has order \(n+1\) and possesses an eigenvalue at one. The last row of the state equation enforces
\[
\sigma \chi_{n+1} = \chi_{n+1}, \qquad \chi_{n+1}(0) = 1,
\]
which defines a constant internal signal \(\chi_{n+1} \equiv 1\). In other words,  \(\mathcal{B}_{\mathrm{lift}}\) embeds an \emph{internal model} of a constant, thus structurally enforcing the presence of an offset in a higher-order linear representation.

\section{Example} \label{sec:example}

Consider an \gls{ATI} system described by the state-space representation~\eqref{eq:system_nonlinear_linearization_affine}, with
\[
\begin{aligned}
A &= 
\begin{bmatrix}
1 & 0 \\
0 & 2
\end{bmatrix}, &
B &= 
\begin{bmatrix}
1 \\ 1
\end{bmatrix}, &
E &= 
\begin{bmatrix}
1 \\ 1
\end{bmatrix}.
\end{aligned}
\]
The system is of order $n=2$ with a scalar input ($m=1$). We illustrate that Theorem \ref{thm:fundamental_lemma_state-space_ATI_systems} requires less data than Theorem \ref{thm:martinelli22} to satisfy the rank condition~\eqref{eq:rank_condition_affine}. We also show that the persistence of excitation condition used in Theorem~\ref{thm:fundamental_lemma_state-space_ATI_systems} is sufficient but not necessary.
For simplicity, we fix the trajectory length to~$L = 2$.

First, we generate a scalar input sequence $u$ of length ${T = 9}$. 
By Theorem~\ref{thm:martinelli22}, the value $T_L(\mathcal{L}) = 9$ is the minimal length for an input sequence $u$ to be persistently exciting for the model class~$\mathcal{L}$ of order ${n + L + 1 = 5}$ (see also Section~\ref{sec:discussion}). The sequence is constructed as a normalized sum of a constant offset and of two sinusoids with random amplitudes and phases, yielding
\[
\begin{aligned}
u(1) &= 0.91,&\; u(2) &= 0.41,&\; u(3) &= -0.53,\\
u(4) &= -0.99,&\; u(5) &= -0.65,&\; u(6) &= 0.20,\\
u(7) &= 0.87,&\; u(8) &= 0.97,&\; u(9) &= 0.32.
\end{aligned}
\]
We simulate the system from zero initial conditions, obtaining
\[
\begin{aligned}
x(1) &= 0,&\; x(2) &= 0,&\; x(3) &= 0.28,\\
x(4) &= 1.30,&\; x(5) &= 1.67,&\; x(6) &= 2.96,\\
x(7) &= 3.96,&\; x(8) &= 4.24,&\; x(9) &= 5.16.
\end{aligned}
\]
A direct computation shows that the rank condition~\eqref{eq:rank_condition_affine} holds.

Next, we generate a scalar input sequence $u$ of length ${T = 7}$.
By Theorem~\ref{thm:fundamental_lemma_state-space_ATI_systems}, the value ${T_L(\mathcal{A}) = 7}$ is the minimal data length required  for $u$ to be persistently exciting for the model class~$\mathcal{A}$ of order ${n + L = 4}$ (see also Section~\ref{sec:discussion}). The sequence is constructed as a normalized sum of two sinusoids with random amplitudes and phases (with no constant offset), yielding
\[
\begin{alignedat}{8}
u(1)&=~0.640,&\ u(2)&=-0.323,&\ u(3)&=-1,&\ u(4)&=~0.248,\\
u(5)&=-0.640,&\ u(6)&=~0.323,&\ u(7)&=~1,&\ u(8)&=-0.248.
\end{alignedat}
\]
Simulating the system from zero initial conditions produces
\[
\begin{alignedat}{8}
x(1)&=0,&\ x(2)&=1.640,&\ x(3)&=2.317,&\ x(4)&=2.317,\\
x(5)&=3.565,&\ x(6)&=3.925,&\ x(7)&=5.248,&\ x(8)&=7.248.
\end{alignedat}
\]
A direct computation shows that the rank condition~\eqref{eq:rank_condition_affine} holds.
As discussed in Section~\ref{sec:discussion}, this illustrates that taking into account the affine structure of the system  leads to a sample complexity reduction of ${T_{n+L+1}(\mathcal{L}) - T_{n+L}(\mathcal{A}) = 2}$.

Finally, we generate a scalar input sequence $u$ of length $T = 6$. The sequence is constructed as a \textit{single} sinusoid with random amplitude and phase, yielding
\[
\begin{aligned}
u(1) &= 1,&\; u(2) &= -0.12,&\; u(3) &= -1,\\
u(4) &= 0.12,&\; u(5) &= 1,&\; u(6) &= -0.12.
\end{aligned}
\]
Simulating the system from zero initial conditions produces
\[
\begin{aligned}
x(1) &= 0,&\; x(2) &= 2,&\; x(3) &= 2.88,\\
x(4) &= 2.88,&\; x(5) &= 4,&\; x(6) &= 6 .
\end{aligned}
\]
A direct computation again confirms that the rank condition~\eqref{eq:rank_condition_affine} holds, showing that the persistence of excitation condition used in Theorem~\ref{thm:fundamental_lemma_state-space_ATI_systems} is sufficient, but not necessary.

\section{Conclusion} \label{sec:conclusion}

This paper has studied the interplay between the model classes of linear and affine time-invariant systems within the framework of behavioral system theory. \gls{ATI} behaviors have been characterized through several representations, illustrating how structural parallels with \gls{LTI} systems enable meaningful extensions of standard notions such as controllability and integer invariants.
A structurally justified persistence of excitation condition for affine systems has also been introduced. The condition reflects the geometry of affine behaviors and yields a fundamental lemma that parallels the one for \gls{LTI} systems while reducing data requirements relative to existing results. Our findings emphasize that excitation conditions must be adapted to the model class, as overlooking structural differences leads to unnecessarily conservative data requirements.

This work contributes to the growing body of research extending behavioral systems theory beyond the model class of \gls{LTI} systems.
Many practically relevant systems—including positive, hybrid, and stochastic systems—generate trajectory spaces that lie outside the scope of existing \gls{LTI} results.
The central challenge is to delineate tractable model classes by analyzing how assumptions on trajectory spaces interact with system structure—an interaction that, in system identification and control, naturally manifests through data informativity and representation theorems.

\appendix

\section{Equivalent affine kernel representations}\label{sec:equivalence}

This section establishes the equivalence relation among affine kernel representations, identifying all polynomial matrices and offset vectors that generate the same system. By Theorem~\ref{thm:imp rep} a system $\mathcal{B}\in\mathcal{A}^q$ is uniquely determined by a polynomial matrix $R(\xi)$ and vector $c$. Here, we will characterize all pairs $R(\xi)$ and $c$ which define the same system. As a first observation, we state and prove the following lemma: 
\begin{lemma}[Unimodular invariance]\label{lem:unimod transform}
Let $\mathcal{B} = \Ker_c R(\sigma)$, with $R(\xi) \in \mathbb{R}^{g\times q}[\xi]$ and $c \in \mathbb{R}^{g}$. 
Let $U(\xi) \in \mathbb{R}^{g\times g}[\xi]$, and define 
\[
R'(\xi) = U(\xi)R(\xi), \qquad 
c' = U(1)c, \qquad 
\mathcal{B}' = \Ker_{c'} R'(\sigma).
\]
Then $\mathcal{B} \subseteq \mathcal{B}'$. Moreover, if $U(\xi)$ is unimodular, then $\mathcal{B} = \mathcal{B}'$.
\end{lemma} 
\begin{proof} 
For any constant trajectory $c$, we have $U(\sigma)c = U(1)c$. 
Therefore, if $R(\sigma)w = c$, pre-multiplying by $U(\sigma)$ immediately yields the first claim. 
For the second part, recall that a unimodular matrix $U(\xi)$ admits a polynomial inverse $U(\xi)^{-1}$. 
Consequently, applying the same argument to $U(\xi)^{-1}$ establishes that ${\mathcal{B}' \subseteq \mathcal{B}}$.
\end{proof}	

We can now generalize \cite[Corollary 3.6.3]{willems1997introduction}, and characterize all affine kernel representations of an \gls{ATI} system. 

\begin{theorem}[Equivalent affine kernel representations] \label{thm:equiv}
	Consider a non-empty system ${\mathcal{B}\in\mathcal{A}^q}$. Assume $\mathcal{B} = \Ker_c  R(\sigma)$, 
	with $R(\xi)\in  \mathbb{R}^{g\times q}[\xi] $ and $c\in\mathbb{R}^{g}$. Let $R'(\xi) \in  \mathbb{R}^{g'\times q}[\xi]$ and $c'\in \mathbb{R}^{g'}$. Then $\mathcal{B}= \Ker_{c'}  R'(\sigma)$ if and only if $R(\xi)$ and $c$ can be transformed into $R'(\xi)$ and $c'$ by repeatedly 
	\begin{itemize}
		\item pre-multiplying $R(\xi)$ both with a unimodular polynomial matrix $U(\xi)$ and $c$ with $U(1)$, or
		\item adding/removing zero-rows from both $R(\xi)$ and $c$. 
	\end{itemize}
\end{theorem} 
\begin{proof} 
	($\Leftarrow$) The implication follows immediately by Lemma~\ref{lem:unimod transform}, since neither operation changes the induced system.
    
    ($\Rightarrow$) Since $R(\xi)$ is a kernel representation of $\mathcal{B}_{\textup{dif}}$, there exists a unimodular matrix $U_1(\xi)$ such that 
	\[ U_1(\xi)R(\xi) = \begin{bmatrix} R_1(\xi) \\ 0 \end{bmatrix},\]
	where $R_1(\xi)$ has full row rank $g_1$. Define 
    $$U_1(1)c = \begin{bmatrix} c_1 \\ c_1'\end{bmatrix},$$ where $c_1\in\mathbb{R}^{g_1}$. Using Lemma~\ref{lem:unimod transform}, we see that 
	\[\mathcal{B} = \left\lbrace w \mid \begin{bmatrix} R_1(\xi) \\ 0 \end{bmatrix}w = \begin{bmatrix} c_1 \\ c_1'\end{bmatrix}\right\rbrace. \] 
	By assumption, $\mathcal{B}$ is non-empty,  so $c_1'=0$. Since removing zero-rows does not change the system, we can conclude that  
	\[ \mathcal{B} = \left\lbrace w \mid R_1(\xi)w = c_1 \right\rbrace. \] 
	Performing the same procedure \textit{mutatis mutandis} for $$\mathcal{B}=\{ w \mid R'(\sigma)w = c'\},$$ 
    we obtain that 
	\[\mathcal{B} = \left\lbrace w \mid R_2(\xi)w = c_2 \right\rbrace, \] 
	where $R_2(\xi)$ has full row rank $g_2$ and $c_2\in\mathbb{R}^{g_2}$. Note that $R_1(\xi)$ and $R_2(\xi)$ induce minimal kernel representations of $\mathcal{B}_{\textup{dif}}$. Thus $g_1=g_2$. Moreover, we can use \cite[Corollary 3.6.3]{willems1997introduction} to conclude that there exists a unimodular $U_3(\xi)\in\mathbb{R}^{g_1\times g_1}[\xi]$ such that $R_2(\xi)= U_3(\xi)R_1(\xi)$. Let $w\in\mathcal{B}$, then by definition $c_2= R_2(\sigma)w = U_3(\sigma)R_1(\sigma)w = U_3(\sigma)c_1= U_3(1)c_1$. 
	
	Now note that 
	$$
    {R'(\xi) = U_2^{-1}(\xi) \begin{bmatrix} R_2(\xi) \\ 0 \end{bmatrix}, }
    \quad 
    {\begin{bmatrix} R_1(\xi) \\ 0 \end{bmatrix} = U_1(\xi)R(\xi),} 
    $$
	and 
	\[ c' = U_2^{-1}(1) \begin{bmatrix} c_2 \\ 0 \end{bmatrix}, \quad c_2 = U_3(1) c_1, \quad \begin{bmatrix} c_1 \\ 0 \end{bmatrix} = U_1(\xi)c, \]
	thus proving the theorem. 
\end{proof}

\section{Consistency of affine kernel representations} 
\label{sec:consistency}

In this section, we consider the \emph{consistency problem} for affine kernel representations of \gls{ATI} systems.
Given a sequence ${c \in (\R^{g})^{\Z}}$ and a polynomial matrix $R\in\mathbb{R}^{g\times q}[\xi]$, the problem consists in determining if there exists a trajectory ${w \in (\R^{q})^{\Z}}$ such that
\beq \label{eq:affine_behavior}
   R(\sigma)w = c .
\eeq
in which case the corresponding \gls{ATI} behavior $\Ker_c \, R(\sigma)$ is said to be \textit{consistent}.  Equivalently, consistency amounts to deciding whether $c$ belongs to the \emph{image} of the polynomial operator defined by matrix polynomial $R(\xi)$, defined as
\[
   \Image \, R(\sigma) := \{\, v \in (\R^{g})^{\Z} : \, \exists \, w \in (\R^q)^{\Z},\; v = R(\sigma)w \,\}.
\]   
The consistency problem for affine kernel representations is the dynamic analogue of the classical question of determining if a linear system of equations is \textit{consistent}, \ie, if
\beq \label{eq:linear_system}
   A x = b
\eeq
with ${A\in \R^{m \times n}}$ and ${b\in \R^{m}}$, admits a solution ${x\in \R^{n}}$. In the static case, the consistency problem is solved by the Rouch\'e--Capelli theorem: the system\,\eqref{eq:linear_system} is consistent if and only if
\beq \label{eq:rouche-capelli}
   \rank\, A  = \rank\,\bigl[\,A\;\;b\,\bigr].
\eeq
or, equivalently, if and only if $y^\top b = 0$ for all $y \in \ker(A^\top)$. In the dynamic case, consistency depends not on constant row dependencies of $R$, but on \emph{polynomial} dependencies expressed by \emph{syzygies}~\cite{eisenbud1995commutative}. Formally, a polynomial row vector ${\lambda(\xi)\in\R[\xi]^{1\times p}}$ is said to be a \emph{(left) syzygy} of $R(\xi)$ if
\[
   \lambda(\xi)\,R(\xi) = 0.
\]
The set of all such polynomial row vectors defines a submodule
\[
   \syz R(\xi) := \{\, \lambda(\xi)\in\R[\xi]^{1\times p} : \lambda(\xi)R(\xi)=0 \,\} \, \subseteq\, \R[\xi]^{1\times p}.
\]
Consistency of~\eqref{eq:affine_behavior} requires that every polynomial dependency among the rows of $R(\xi)$ is satisfied also by $c$, as formalized by the following result.

\begin{theorem}[Syzygy criterion for consistency]
\label{thm:syzygy-feasibility}
Consider a polynomial matrix  ${R(\xi)\in\R[\xi]^{g\times q}}$ and a sequence ${c \in (\R^{g})^{\Z}}$. Then ${c \in \Image \, R(\sigma)}$ if and only if ${\lambda(\sigma)\,c = 0}$ for all ${\lambda(\xi)\in\syz R[\xi]}$.
\end{theorem}

\begin{proof}
($\Rightarrow$) If $c=R(\sigma)w$ for some $w \in (\R^q)^{\Z}$, then
\[
   \lambda(\sigma)c=\lambda(\sigma)R(\sigma)w=0
\]
for all $\lambda(\xi)\in\syz R[\xi]$.

($\Leftarrow$) Suppose ${\lambda(\sigma)c=0}$ for all ${\lambda(\xi)\in\syz R[\xi]}$.  If ${R(\xi)}$ is full row rank or ${R(\xi)=0}$, the claim holds trivially. Otherwise, consider the Smith decomposition
\[
   U(\xi)\,R(\xi)\,V(\xi)
   = \begin{bmatrix} D(\xi) & 0 \\ 0 & 0 \end{bmatrix}
   =:\widetilde R(\xi),
\]
where $U(\xi)\in\mathbb{R}^{g\times g}[\xi]$ and $V(\xi)\in\mathbb{R}^{q\times q}[\xi]$ are unimodular and 
$$D(\xi) = \diag(d_1(\xi),\dots,d_r(\xi)),$$ with ${d_i(\xi)\neq 0}$. 
Let 
$$\tilde c=U(\sigma)c.$$  
The assumption $\lambda(\sigma)c=0$ ensures that the lower block of $\tilde c$ vanishes, \ie, $\tilde c=
\left[\begin{smallmatrix} c_1\\ 0 \end{smallmatrix}\right]$, with $\tilde c_1\in (\R^r)^{\Z}$. Thus, consistency reduces to the possibility of solving $D(\sigma)z_1=\tilde c_1$ for some $z_1\in (\R^r)^{\Z}$. Each $d_i(\sigma)$ acts diagonally on the corresponding component, reducing the problem to $r$ scalar difference equations of the form $d_i(\sigma) z_i = \tilde c_i$, which always admit a solution since ${d_i \neq 0}$.  Choosing $z_1$ accordingly and arbitrary free variables $z_2$, one obtains $z= 
\left[\begin{smallmatrix} z_1 \\ z_2 \end{smallmatrix}\right]$, and then $w:=V(\sigma)z$ satisfies $R(\sigma)w=c$. Therefore, $c\in\Image\, R(\sigma)$.
\end{proof}

\begin{remark}[Computation of syzygies]
In the static case, the Rouch\'e-Capelli rank condition~\eqref{eq:rouche-capelli} yields a direct test implementable with standard linear algebra. By contrast, in the (possibly infinite-dimensional) dynamic case, one must appeal to tools from commutative algebra~\cite{eisenbud1995commutative}. From a module-theoretic perspective, the rows of \(R(\xi)\) span a submodule of \(\R[\xi]^{1\times q}\), and $\syz R[\xi]\subseteq\R[\xi]^{1\times p}$ consists of all polynomial relations among these generators. Because \(\R[\xi]\) is a Euclidean domain (hence, a Noetherian ring\footnote{A \emph{Euclidean domain} is a ring admitting a division algorithm. Every Euclidean domain is a \emph{Noetherian ring}, meaning it satisfies the ascending chain condition on ideals (or, equivalently, all ideals its are finitely generated). Over a Noetherian ring \(R\), every submodule of a finitely generated \(R\)-module is finitely generated~\cite{eisenbud1995commutative}. Thus, for \(R=\R[\xi]\), the syzygy module \(\syz R[\xi]\) is finitely generated.}), \(\syz R[\xi]\) is finitely generated. Thus, generators of \(\syz R[\xi]\) can be computed via Gr\"obner bases  (or, alternatively, by means of the Smith–Popov form). Given a generating set \(\{\lambda_i(\xi)\}_{i\in\mathbf{s}}\) of \(\syz R[\xi]\), consistency of~\eqref{eq:affine_behavior} is equivalent to the finite family of polynomial constraints $\lambda_i(\xi)\,s = 0$, with $i\in\mathbf{s}.$
\end{remark}

\begin{example}[Consistency of affine kernel representations]
Consider the polynomial matrix
\[
  R(\xi) \;=\;
  \begin{bmatrix}
     1 & 0 & 0 \\[3pt]
     0 & 1-\xi & 0
  \end{bmatrix} \in \R[\xi]^{2\times 3}.
\]
The matrix $R(\xi)$ is full row rank over $\R[\xi]$, so  $\syz R(\xi)=\{0\}$.  
Consequently, any given $c \in (\R^{2})^{\Z}$ gives rise to a consistent affine kernel representation $\Ker_c R(\xi)$. By contrast, consistency imposes trajectory constraints whenever there exist non-trivial syzygies.   For example, consider the polynomial matrix
\[
R(\xi)=
\begin{bmatrix}
\xi+1 & \xi & \xi+2\\[2pt]
\xi^2-1 & \xi^2-\xi & \xi^2+\xi-2
\end{bmatrix}\in\R[\xi]^{2\times 3}.
\]
The matrix $R(\xi)$ is \textit{not} full row rank over $\R[\xi]$, since 
\[
\lambda(\xi)=[\,1-\xi\quad 1\,]\in \syz R(\xi) .
\]
Accordingly, any given $c \in (\R^{2})^{\Z}$ is consistent if and only if 
\[
c_2=(\xi-1)c_1 .
\]
For instance, if $c_1(t)=(-1)^t$, then 
$$c_2=(\xi-1)c_1=(-1)^{t+1}-(-1)^t=-2(-1)^t,$$ 
so the corresponding affine kernel representation $\Ker_c R(\xi)$ is consistent. Instead, if  $c_1\equiv 0$ and $c_2\equiv 1$, the constraint $$(\xi-1)c_1-c_2= -1\not\equiv 0$$ 
is not satisfied and, hence, $c$ is and the corresponding affine kernel representation $\Ker_c R(\xi)$ is empty.
\end{example}

Feasibility can be reduced to finite-horizon conditions by considering block-Toeplitz truncations of $R(\sigma)$ and applying rank tests. Let $T\in\N$ and $R(\xi)=\sum_{i=0}^{d}R_i\xi^i$, with $R_i\in\R^{g\times q}$.  Define the block-Toeplitz matrix
\[
  R_T =
  \begin{bmatrix}
    R_0 & R_1 & \cdots & R_d &        & 0\\
    0   & R_0 & \cdots & R_{k-1} & \ddots & \vdots\\
    \vdots &   & \ddots &        & \ddots & R_k\\
    0 & \cdots & 0 & R_0 & \cdots & R_d
  \end{bmatrix}
  \in  \R^{gT\times q(T+k)} .
\]
By restricting~\eqref{eq:affine_behavior} to the time interval $[1,T+d]$ one obtains
\beq \label{eq:RC_trajectories_system}
   R_T\,w|_{[1,T+d]} \;=\; c|_{[1,T]} .
\eeq
By the Rouch\'e--Capelli theorem, system~\eqref{eq:RC_trajectories_system} is consistent if and only if
\beq \label{eq:RC_trajectories}
   \rank \, R_T  = \rank\,[\,R_T \;\; c|_{[1,T]}\,]\,.
\eeq 

\begin{theorem}[Horizon bound via maximal syzygy degree]
\label{thm:horizon}
Consider a polynomial matrix  ${R(\xi)\in\R[\xi]^{g\times q}}$ of degree $d$ and a sequence $c \in (\R^{g})^{\Z}$. Assume there exists
 a generating set \(\{\lambda_i(\xi)\}_{i\in\mathbf{s}}\) of \(\syz R[\xi]\) such that ${\deg \lambda_i(\xi)\le \delta}$ for all  $i\in\mathbf{s}$. 
Then ${c\in\Image\,R(\sigma)}$ if and only if, for all ${\tau\in\Z}$ and $T\ge \delta+1$, there exists $w \in (\R^q)^{\Z}$ such that
\beq \label{eq:shifts}
   R_T\,w|_{[\tau+1,\,\tau+T+d]} \;=\; c|_{[\tau+1,\,\tau+T]} .
\eeq
\end{theorem}

\begin{proof}
$(\Rightarrow)$: Immediate (by restriction).

$(\Leftarrow)$: Fix $\lambda(\xi)\in\syz R[\xi]$ with $\deg\lambda(\xi)\le\delta$. 
For each $t\in\Z$, applying~\eqref{eq:shifts} with $\tau=t-1$ and $T=\delta$ yields 
a finite sequence $w^{(t)}|_{[t,\,t+d+\delta]}$ such that
\[
   R_T\,w^{(t)}|_{[t,\,t+d+\delta]} = c|_{[t,\,t+\delta]} .
\]
Because $\lambda(\xi)R(\xi)=0$, one obtains
\[
   (\lambda(\sigma) c)(t) = (\lambda(\sigma) R(\sigma) w^{(t)})(t) = 0 .
\]
The assumption that~\eqref{eq:shifts} holds for all time shifts ensures that overlapping 
sequences $w^{(t)}$ and $w^{(t+1)}$ coincide on their common support, allowing a consistent 
concatenation into a sequence $w\in(\R^q)^{\Z}$ satisfying $R(\sigma)w=c$. 
Thus $\lambda(\sigma)c\equiv0$ for all $\lambda(\xi)\in\syz R(\xi)$, and by 
Theorem~\ref{thm:syzygy-feasibility}, $c\in\Image\,R(\sigma)$.
\end{proof}

\subsection{Proofs}  \label{ssec:proofs}

\begin{proof}[Proof of Theorem~\ref{thm:imp rep}]
${(\Leftarrow):}$ Assume $R(\xi)$ defines a kernel representation for $\B_\textup{dif}$ as in~\eqref{eq:linear_kernel_representation}. Fix ${w_1,w_2\in\B}$ and define $$w_\textup{dif} :=w_1-w_2.$$
     By construction, $w_\textup{dif} \in\B_\textup{dif}.$ Moreover, we have 
     $$R(\sigma)w_1 = R(\sigma)w_\textup{dif}+R(\sigma)w_2 = R(\sigma)w_2. $$ 
     Since $w_1,w_2 \in \B$ are arbitrary and $\B$ is time-invariant, we may choose ${w_2= \sigma^{k} w_1}$, with $k\in\Z$. This yields
     $$R(\sigma)w_1= \sigma^k R(\sigma)w_1, \quad \forall \, k\in\Z, $$ 
     which implies that ${R(\sigma)w_1= c}$ for some constant $c\in\mathbb{R}^p$. For this choice of $c$, we conclude that 
     $$\B \subseteq\{ w \mid R(\sigma) w = c \}.$$ 
     Now let $w_3$ be such that $R(\sigma) w_3 = c $. Then 
     $$w_3-w_1\in \ker R(\sigma) =\B_\textup{dif},$$ 
     and, hence, $w_3 \in \B$. This proves the inclusion 
     $$\{ w \mid R(\sigma) w = c \}\subseteq \B$$ 
     and, hence, ~\eqref{eq:affine_kernel_representation} holds. 
    
${(\Rightarrow):}$ Let $c\in\mathbb{R}^p$ and let $R(\xi)$ be such that \eqref{eq:affine_kernel_representation} holds. Let $w_\textup{dif}\in\B_\textup{dif}$, that is, $w_\textup{dif}=w_1-w_2$ with $w_1,w_2\in\B$. Then 
$$R(\sigma)w= R(\sigma)(w_1-w_2) = c-c = 0,$$ 
and, hence, 
$$\B_\textup{dif}\subseteq\{ w \mid R(\sigma) w =0 \}.$$
Now take $w$ such that $R(\sigma)w= 0$. For any $w_1\in\B$, we have
$$R(\sigma)w+w_1 = R(\sigma)w_1 = c.$$ 
Thus, $w+w_1\in \B$ and $w_1 \in \B$. Consequently, $w\in \B_\textup{dif}$. Then 
$$\{ w \mid R(\sigma) w =0 \}\subseteq \B_\textup{dif},$$ which concludes the proof. 
\end{proof}

\begin{proof}[Proof of Theorem~\ref{thm:io}] Let $\wo\in\B$. By Theorem~\ref{thm:lin+offset}, we infer
$$\pi_u \B = \pi_u \B_\textup{dif} + \pi_u \wo
$$
and
$$
\pi_y \B = \pi_y \B_\textup{dif} + \pi_y \wo.$$ 
Note that $u$ is free for $\B$ if and only if 
$$\pi_u \B_\textup{dif} + \pi_u \wo=(\mathbb{R}^{m})^\Z.$$ 
Since $\pi_u \B_\textup{dif}$ is a subspace, and $\pi_u \wo$ is a sequence, this holds if and only if 
$$\pi_u \B_\textup{dif}=(\mathbb{R}^{m})^\Z.$$ 
Thus, $u$ is free for $\B$ if and only if it is free for $\B_\textup{dif}$, from which we  conclude the claim. \end{proof}

\begin{proof}[Proof of Theorem~\ref{thm:affine states vs linear states}]
$(\Leftarrow)$ Let $w_1\in\B$. By time-invariance, $\sigma w_1\in\B$. Then 
$$\hat{w} := w_1-\sigma w_1 \in \mathcal{B}_{\textup{dif}}.$$ 
By \eqref{eq:linear state}, there exists $\hat{x}$ such that 
$$\sigma \hat{x} = A\hat{x} +B\hat{u} \ \text{ and } \ \hat{y}= C\hat{x}+D\hat{u}.$$ 
Now fix $x_0\in\mathbb{R}^n$, and define $x_1$ through 
$$\sigma x_1= x_1-\hat{x},$$ with $x_1(0)=x_0$. Then
\beq \nn
\begin{bmatrix} A-\sigma I & B & 0 \\ C & D & -I \end{bmatrix}\begin{bmatrix} x_1 - \sigma x_1 \\ w_1-\sigma w_1 \end{bmatrix} = 0 
\eeq  
\noindent
implies  
\beq \nn
\begin{bmatrix} A-\sigma I & B & 0 \\ C & D & -I \end{bmatrix}\begin{bmatrix} x_1  \\ w_1 \end{bmatrix} = \sigma\begin{bmatrix} A-\sigma I & B & 0 \\ C & D & -I \end{bmatrix}\begin{bmatrix} x_1  \\ w_1 \end{bmatrix} . 
\eeq
\noindent
This shows that the right-hand side is constant in time. As a consequence, there exist $E\in\mathbb{R}^n$ and $F\in\mathbb{R}^p$ such that  
\[\sigma {x}_1 = A{x}_1 +B{u}_1+E \quad \textrm{ and } \quad {y}_1= C{x}_1+D{u}_1+F. \]  
\noindent
Note that $E$ and $F$ depend on the construction above. We will show that this specific choice allows us to prove \eqref{eq:affine_state_space_representation}. We will prove this by mutual inclusion. 

Let $w_2\in\B$, and let $w= w_1-w_2\in\mathcal{B}_{\textup{dif}}$. By \eqref{eq:linear state}, there exists ${x}$ such that 
$$\sigma {x} = A{x} +B{u} \ \text{ and } \ {y}= C{x}+D{u}.$$ 
Now consider the previously constructed element ${x}_1$ and define $${x}_2:={x}_1-x .$$ 
Then 
\[
\sigma {x}_2 = A{x}_2 +B{u}_2+E \quad \textrm{ and } \quad {y}_2= C{x}_2+D{u}_2+F , \] 
which shows the first inclusion towards proving \eqref{eq:affine_state_space_representation}. 

Now consider $w_3, x_3$ such that 
\[\sigma {x}_3 = A{x}_3 +B{u}_3+E \quad \textrm{ and } \quad {y}_3= C{x}_3+D{u}_3+F. \] 
Using \eqref{eq:linear state}, one obtains $w_3-w_1\in \mathcal{B}_{\textup{dif}}$. By Theorem~\ref{thm:lin+offset}, this implies $w_3 \in \B$. Thus, this proves that \eqref{eq:affine_state_space_representation} holds for the choice of $E$ and $F$. 

$(\Rightarrow)$  In order to prove equality in \eqref{eq:linear state}, we will prove mutual inclusion. 

Let $w\in \B_\textup{dif}$. Note that $w=w_1-w_2$ for some $w_1,w_2\in\B$. Using \eqref{eq:affine_state_space_representation}, there exist $x_1,x_2$ such that $\sigma {x}_i = A{x}_i +B{u}_i+E$ and ${y}_i= C{x}_i+D{u}_i+F$. Now let $x:= x_1-x_2$ and note that $\sigma x = Ax+Bu$ and $y=Cx+Du$. 

We turn our attention to the other inclusion. Let $w$ and $x$ be such that $\sigma x = Ax+Bu$ and $y=Cx+Du$. Let $w_1\in \B$. By \eqref{eq:affine_state_space_representation}, there exists $x_1$ such that $\sigma {x}_1 = A{x}_1 +B{u}_1+E$ and ${y}_1= C{x}_1+D{u}_1+F$. Define $w_2:= w_1-w$ and $x_2:=x_1-x$. Clearly $\sigma {x}_2 = A{x}_2 +B{u}_2+E$ and ${y}_2= C{x}_2+D{u}_2+F$, and hence, by \eqref{eq:affine_state_space_representation}, $w_2\in \B$. Now $w= w_2-w_1$, and thus $w\in\B_\textup{dif}$. This, in turn, completes the proof.
\end{proof}

\begin{proof}[Proof of Lemma\ref{lemma:controllability_is_translation-invariant}]
Let $\mathcal{B}$ be controllable and fix ${\wo \in (\mathbb{R}^q)^{\mathbb{Z}}}$. Given any $w_1 + \wo, w_2 + \wo \in \mathcal{B} + \{ \wo \}$ and any $t_0 \in \N$, there exist $w_1, w_2 \in \mathcal{B}$ such that $w_i + \wo \in \mathcal{B} + \{ \wo \}$. By controllability, there exists ${w} \in \mathcal{B}$ and $t_1 \ge t_0$ such that ${w}$ patches $w_1$ and $w_2$ as in~\eqref{eq:cont patch}. Then ${w} + \wo$ patches $w_1 + \wo$ and $w_2 + \wo$ and lies in $\mathcal{B} + \{ \wo \}$. The converse follows by translating with $-\wo$.
\end{proof}

\begin{proof}[Proof of Lemma~\ref{lemma:integer-invariants}] 
First, observe that the input cardinality, state cardinality, and lag of $\B$ and $\B_{\textup{dif}}$ are equal as a consequence of Theorem~\ref{thm:io}, Theorem~\ref{thm:affine states vs linear states}, and Theorem~\ref{thm:imp rep} respectively. Moreover, for all $t \in \N$, we have
$$\dim \B|_{[1,t]} = \dim \B_{\textup{dif}}|_{[1,t]}.$$
Since $\B_{\textup{dif}}\in\L^q$, the claim follows directly from arguments analogous to those in the proof of~\cite[Theorem 6]{willems1986timea}.
\end{proof} 

\begin{proof}[Proof of Theorem~\ref{thm:fundamental_lemma_state-space_ATI_systems} ] The proof is inspired by that of~\cite[Theorem~1]{vanWaarde2020willems}.

(i). We need to show that the rank condition~\eqref{eq:rank_condition_aff} holds. By contradiction, assume 
\beq \nn
    \rank \, 
    \scalebox{1}{$
    \bma  
    \begin{array}{c}
    H_1(x_d) \\
    H_L(u_d) \\ 
    \mathbb{1}^{\transpose} 
    \end{array}
    \ema < mL + n + 1$} .
\eeq
Then, there exist \(\nu \in \mathbb{R}^{1 \times n}\), \(\eta \in \mathbb{R}^{1 \times mL}\) and \(\epsilon \in \mathbb{R}\) not all zero and such that 
\beq \nn
    \begin{bmatrix} 
    \nu & \eta & \epsilon 
    \end{bmatrix}
    \scalebox{1}{$
    \bma  
    \begin{array}{c}
    H_1(x_d) \\
    H_L(u_d) \\ 
    \mathbb{1}^{\transpose} 
    \end{array}
    \ema = 0$} .
\eeq
Since  
\(u_d \in (\R^m)^{\mathbf{T}}\)
and
\(x_d \in (\R^n)^{\mathbf{T}}\)
is an input/state trajectory of system~\eqref{eq:affine_state_space_representation}, by repeatedly applying the update laws of the system~\eqref{eq:state-space-affine}, we obtain
\begingroup
\setlength{\arraycolsep}{1pt}   
\renewcommand{\arraystretch}{1.15} 
\beq \label{eq:FL_1-a} 
\bma
\begin{array}{c|c}
\begin{array}{cccccc}
\nu & \eta & 0 & \cdots & 0  \\
\nu A & \nu B & \eta & \cdots & 0  \\
\nu A^2 & \nu AB & \nu B & \cdots & 0   \\
\vdots & \vdots & \vdots & \ddots & \vdots  \\
\nu A^n & \nu A^{n-1} B & \nu A^{n-2} B & \cdots & \eta 
\end{array} &
\begin{array}{c}
\lambda
\end{array}
\!
\end{array}
\ema
\!\!
\bma  
    \begin{array}{c}
    H_1(x_d) \\ 
    {H}_{n+L}(u_d)\\  \hline
    \mathbb{1}^{\transpose} 
    \end{array}
\ema
\! = 0,
\eeq
\endgroup
where $\lambda \in \R^{n+1}$ is defined as
\beq \label{eq:FL_3-a}
\lambda =
\begin{bmatrix}
\epsilon \\
\nu E + \epsilon \\
\nu (I + A)E + \epsilon \\
\vdots \\
\nu \biggl( \sum_{i=0}^{n-1} A^i \biggr) E + \epsilon
\end{bmatrix} 
\eeq
Now, let \(\beta = \begin{bmatrix} \beta_0 & \cdots & \beta_n \end{bmatrix}^\top \in \mathbb{R}^{n+1}\), with \(\beta_n = 1\), be the vector of coefficients of the characteristic polynomial of the matrix $A$. 
By the Cayley-Hamilton theorem, we have 
\beq \label{eq:cayley-hamilton}
\beta_0 I + \beta_1 A + \cdots + \beta_n A^n  = 0.
\eeq
Thus, multiplying both sides of~\eqref{eq:FL_1-a} by the row vector $\beta^{\top}$ and invoking~\eqref{eq:FL_3-a} together with~\eqref{eq:cayley-hamilton}, one obtains 
\begingroup
\setlength{\arraycolsep}{2.pt}   
\renewcommand{\arraystretch}{1.15} 
\beq \label{eq:FL_2-a}
\! \! \! \! 
\begin{bmatrix} 0 & \beta_0 \eta + \nu \sum_{i=1}^n \beta_i A^{i-1} B & \cdots & \eta & \beta^\top \lambda \end{bmatrix}
\! 
 \begin{bmatrix} {H}_{n+L}(u_d) \\
    \mathbb{1}^{\transpose}  \end{bmatrix} 
    \! \! = \!  0. \!   \! 
\eeq
\endgroup
Since by assumption the input $u_d$ is persistently exciting of order $n+L$ for the model class $\A$, we have
\beq \label{eq:FL_4-a}
\rank \, 
    \scalebox{1}{$
    \bma  
    \begin{array}{c}
    H_{n+L}(u_d) \\ 
    \mathbb{1}^{\transpose} 
    \end{array}
    \ema = m(n+L) + 1$} .
\eeq
Then \eqref{eq:FL_2-a} and~\eqref{eq:FL_4-a} together imply 
\beq \label{eq:FL_4-b}
\begingroup
\setlength{\arraycolsep}{2.5pt}   
\begin{bmatrix} 0 & \beta_0 \eta + \nu \sum_{i=1}^n \beta_i A^{i-1} B & \cdots & \nu B + \eta  & \eta & \beta^\top \lambda \end{bmatrix} = 0. 
\endgroup
\eeq
Now observe that~\eqref{eq:FL_4-b} immediately implies 
\[
\eta = 0.
\]
Substituting $\eta = 0$ into~\eqref{eq:FL_4-b} yields 
\[
\nu B = 0.
\]
Using this relation, \eqref{eq:FL_4-b} implies 
\[
\nu (\beta_{n-1} B + AB) = 0,
\]
which gives  
\[
\nu AB = 0.
\]
Proceeding recursively, we obtain  
\beq \nn
\nu \begin{bmatrix} B & AB & \cdots & A^{n-1} B \end{bmatrix} = 0.
\eeq
By controllability of the pair $(A,B)$, we conclude that \( \nu = 0 \). Together with~\eqref{eq:FL_3-a} and~\eqref{eq:FL_2-a}, this implies that $\epsilon = 0$. We have thus reached a contradiction, as this would mean that \(\nu \in \mathbb{R}^{1 \times n}\), \(\eta \in \mathbb{R}^{1 \times mL}\), and \(\epsilon \in \mathbb{R}\) are all zero. Therefore, the rank condition \eqref{eq:rank_condition_aff} holds.

(ii). Using the dynamics of the system~\eqref{eq:state-space-affine}, we obtain 
\begin{align*}
    x(2) &= Ax(1) + Bu(1) + E, \\
    x(3) &= A^2 x(1) + ABu(1) + Bu(2) + AE + E, \\
         &~\, \vdots \\
    x(t) &= A^{t-1} x(1) + \sum_{k=0}^{t-2} A^k Bu(t-k-1)  + \sum_{k=0}^{t-2} A^{k} E.
\end{align*} 
and 
\[
y(t)
= C A^{t-1} x(1)
  + \sum_{k=0}^{t-2}\! C A^{k} (B\,u(t-k-1) + E)
  + D u(t) + F .
\]
Consequently, every
trajectory $\left[\begin{smallmatrix} u \\ y\end{smallmatrix}\right] \in \R^{(m+p)\mathbf{L}}$ of
    the system of length \(L\) can be expressed as
\begin{equation}
\begin{bmatrix}
    u \\
    y
\end{bmatrix}
=
\mathsf{M}_L
\begin{bmatrix}\label{eq:every traj is ML times}
    x(1) \\
    u \\
    1
\end{bmatrix}, 
\end{equation}
with $\mathsf{M}_L \in \mathbb{R}^{(m+p)L \times (n +mL+1)}$ defined as
\begingroup
\setlength{\arraycolsep}{2.5pt}   
\renewcommand{\arraystretch}{1.15} 
\begin{equation*} 
\mathsf{M}_L =
\begin{bmatrix}
    \begin{array}{c|c|c}
        0_{mL \times n} & {I}_{mL} & 0_{mL \times 1} \\ 
        \hline
        \mathsf{O}_L(A,C) & \mathsf{T}_L(A,B,C,D) & \mathsf{T}_L(A,E,C,F) \mathbb{1}
    \end{array}
\end{bmatrix}  ,
\end{equation*} 
\endgroup
where $\mathsf{O}_L(A,C) \in \mathbb{R}^{pL \times n} $ and  $\mathsf{T}_L(A,B,C,D) \in \mathbb{R}^{pL \times mL}$
are  the \textit{extended observability matrix} and \textit{convolution matrix}, defined as 
\begin{equation*} 
\mathsf{O}_L(A,C) =
\begin{bmatrix}
    C \\ CA \\ CA^2 \\ \vdots \\ CA^{L-1}
\end{bmatrix} 
\end{equation*} 
and
\begingroup
\setlength{\arraycolsep}{2.pt}   
\renewcommand{\arraystretch}{1.15} 
\begin{equation*} 
\mathsf{T}_L(A,B,C,D) =
\begin{bmatrix}
    D & 0 & 0 & \cdots & 0 \\
    CB & D & 0 & \cdots & 0 \\
    CAB & CB & D & \cdots & 0 \\
    \vdots & \vdots & \vdots & \ddots & \vdots \\
    CA^{L-2}B & CA^{L-3}B & CA^{L-4}B & \cdots & D
\end{bmatrix} ,
\end{equation*}
\endgroup
respectively.

Now recall that, by assumption, the rank condition~\eqref{eq:rank_condition_aff} holds. Then, for every
$x(1)\in\R^n$ and $u\in\R^{mL}$, there exists \( g \in \mathbb{R}^{T - L + 1} \) such that
\begin{equation} \label{eq:every ini in aff span}
\begin{bmatrix}
    x(1) \\
    u \\
    1
\end{bmatrix}
=
\begin{bmatrix}
    H_L(x_d) \\
    H_L(u_d) \\
    \mathbb{1}^\transpose
\end{bmatrix} g
\end{equation}
\noindent
Furthermore, \eqref{eq:every traj is ML times} implies
\begin{equation}\label{eq:Hankels ML times Hankels}
\begin{bmatrix}
    H_L(u_d) \\
    H_L(y_d)
\end{bmatrix}
=
\mathsf{M}_L
\begin{bmatrix}
    H_L(x_d) \\
    H_L(u_d) \\
    \mathbb{1}^\transpose
\end{bmatrix}
\end{equation}
\noindent
Thus, combining~\eqref{eq:every traj is ML times}, ~\eqref{eq:every ini in aff span} and~\eqref{eq:Hankels ML times Hankels}, one obtains that every trajectory ${\left[\begin{smallmatrix} u \\ y\end{smallmatrix}\right] \in \R^{(m+p)\mathbf{L}}}$ of length $L$ can be written as
\[
\begin{bmatrix}
    u \\
    y
\end{bmatrix}
\overset{\eqref{eq:every traj is ML times}}{=}
\mathsf{M}_L
\begin{bmatrix}
    x(1) \\
    u \\
    1
\end{bmatrix}
\overset{\eqref{eq:every ini in aff span}}{=}
\mathsf{M}_L
\begin{bmatrix}
    H_L(x_d) \\
    H_L(u_d) \\
    \mathbb{1}^\transpose
\end{bmatrix} g
\overset{\eqref{eq:Hankels ML times Hankels}}{=}
\begin{bmatrix}
    H_L(u_d) \\
    H_L(y_d)
\end{bmatrix} g,
\]
with \( g \in \mathbb{R}^{T-L+1} \) such that \( \mathbb{1}^{\transpose} g = 1 \). 
\end{proof}

\begin{proof}[Proof of Theorem~\ref{thm:fundamental_lemma_ATI_systems}]
By Theorem~\ref{thm:affine states vs linear states}, $\mathcal{B}$ admits a minimal state space representation~\eqref{eq:state-space-affine}, that is, there exist matrices $A\in \mathbb{R}^{n\times n}, B\in\mathbb{R}^{n\times m}, C\in \mathbb{R}^{p\times n}, D \in \mathbb{R}^{p\times m}$, and vectors $E\in\mathbb{R}^n$, and $F\in\mathbb{R}^p$, with $m=\mathbf{m}(\B)$, $n=\mathbf{n}(\B)$, and $p=\mathbf{p}(\B)$ such that \eqref{eq:affine_state_space_representation} holds. Furthermore, there exists a state trajectory $x_d\in \mathbb{R}^{n\mathbf{T}}$ corresponding to the sequence $\left[\begin{smallmatrix} u_d \\ y_d\end{smallmatrix}\right]$. From this point onward, the argument of the proof follows that of Theorem~\ref{thm:fundamental_lemma_state-space_ATI_systems}.
\end{proof}

\begin{proof}[Proof of Lemma~\ref{lemma:PE-relationship}]
The first claim is immediate, since $H_L(u)$ is a submatrix of 
$$
\bma 
\begin{array}{c}
H_L(u) \\ \mathbb{1}^\top 
\end{array}
\ema ,
$$
and thus has full row rank whenever $u$ is persistently exciting of order $L$ for $\A$.  For the second claim, assume $u$ is persistently exciting of order $L+1$ for $\L$, that is,
\[
   \rank\, H_{L+1}(u) = m(L+1).
\]
Then $\mathbb{1}\notin \, \Image H_L(u)^\top$ in view of~\cite[Proposition.~5]{martinelli2022data}. Then
\[
\rank \, \bma \begin{array}{c} H_{L+1}(u) \\ \mathbb{1}^{\transpose} \end{array} \ema = m(L+1)+1 . \qedhere
\]
\end{proof}

\bibliographystyle{IEEEtran}
\bibliography{refs}

\newcommand{\TAC}{\textit{{IEEE} Trans. Autom.
  Control}}\newcommand{\TCST}{\textit{{IEEE} Trans. Syst.
  Tech.}}\newcommand{\TIT}{\textit{{IEEE} Trans. Inform.
  Theory}}\newcommand{\TSP}{\textit{{IEEE} Trans. Sign.
  Proc.}}\newcommand{\SCL}{\textit{Syst. Control
  Lett.}}\newcommand{\IJC}{\textit{Int. J.
  Control}}\newcommand{\EJC}{\textit{Eur. J.
  Control}}\newcommand{\ACC}{\textit{Proc. Amer. Control
  Conf.}}\newcommand{\ECC}{\textit{Proc. Eur. Control
  Conf.}}\newcommand{\CDC}[1]{\textit{Proc. {#1} Conf. Decision
  Control}}\newcommand{\CDCs}[1]{\textit{{#1} Conf. Decision
  Control}}\newcommand{\IFAC}[1]{\textit{Proc. {#1} IFAC World
  Congr.}}\newcommand{\SIAM}{\textit{SIAM J. Control
  Optim.}}\newcommand{\MCSS}{\textit{Math. Control, Sign.
  Syst.}}\newcommand{\NOLCOS}[1]{\textit{Proc. {#1} IFAC Symp. Nonlinear
  Control Syst.}}\newcommand{\MTNS}[1]{\textit{ Proc. {#1} Math. Symp. Netw.
  Syst.}}
\begin{thebibliography}{10}
\providecommand{\url}[1]{#1}
\csname url@samestyle\endcsname
\providecommand{\newblock}{\relax}
\providecommand{\bibinfo}[2]{#2}
\providecommand{\BIBentrySTDinterwordspacing}{\spaceskip=0pt\relax}
\providecommand{\BIBentryALTinterwordstretchfactor}{4}
\providecommand{\BIBentryALTinterwordspacing}{\spaceskip=\fontdimen2\font plus
\BIBentryALTinterwordstretchfactor\fontdimen3\font minus
  \fontdimen4\font\relax}
\providecommand{\BIBforeignlanguage}[2]{{%
\expandafter\ifx\csname l@#1\endcsname\relax
\typeout{** WARNING: IEEEtran.bst: No hyphenation pattern has been}%
\typeout{** loaded for the language `#1'. Using the pattern for}%
\typeout{** the default language instead.}%
\else
\language=\csname l@#1\endcsname
\fi
#2}}
\providecommand{\BIBdecl}{\relax}
\BIBdecl

\bibitem{markovsky2021behavioral}
I.~Markovsky and F.~D{\"o}rfler, ``Behavioral systems theory in data-driven
  analysis, signal processing, and control,'' \emph{Ann. Rev. Control},
  vol.~52, pp. 42--64, 2021.

\bibitem{willems1986timea}
J.~C. Willems, ``{From time series to linear system—Part {I}. Finite
  dimensional linear time invariant systems},'' \emph{Automatica}, vol.~22,
  no.~5, pp. 561--580, 1986.

\bibitem{willems1986timeb}
------, ``{From Time Series to Linear System-- Part {II}. Exact Modelling},''
  \emph{Automatica}, vol.~22, no.~6, pp. 675--694, 1986.

\bibitem{willems1987timec}
------, ``{From Time Series to Linear System-- Part {II}. Approximate
  modelling},'' \emph{Automatica}, vol.~23, no.~1, pp. 87--115, 1987.

\bibitem{soderstrom1989system}
T.~D. S{\"o}derstr{\"o}m and P.~G. Stoica, \emph{System identification}.\hskip
  1em plus 0.5em minus 0.4em\relax Upper Saddle River, NJ, USA: Prentice-Hall,
  1989.

\bibitem{ljung1999system}
L.~Ljung, \emph{System identification - Theory for the user {(2nd
  edition)}}.\hskip 1em plus 0.5em minus 0.4em\relax Upper Saddle River, NJ,
  USA: Prentice-Hall, 1999.

\bibitem{VanOverschee1996subspace}
P.~{Van Overschee} and B.~{De Moor}, \emph{{Subspace identification for linear
  systems: theory - implementation - applications}}.\hskip 1em plus 0.5em minus
  0.4em\relax Dordrecht, The Netherlands: {Kluwer}, 1996.

\bibitem{willems2005note}
J.~C. Willems, P.~Rapisarda, I.~Markovsky, and B.~L. M.~D. Moor, ``A note on
  persistency of excitation,'' \emph{Systems \& Control Letters}, vol.~54,
  no.~5, pp. 325--329, 2005.

\bibitem{vanWaarde2020willems}
H.~J. van Waarde, C.~{De Persis}, M.~K. Camlibel, and P.~Tesi, ``Willems’
  fundamental lemma for state-space systems and its extension to multiple
  datasets,'' \emph{IEEE Control Systems Letters}, vol.~4, no.~3, pp. 602--607,
  2020.

\bibitem{de2019formulas}
C.~De~Persis and P.~Tesi, ``Formulas for data-driven control: Stabilization,
  optimality, and robustness,'' \emph{\TAC}, vol.~65, no.~3, pp. 909--924,
  2019.

\bibitem{coulson2019data}
J.~Coulson, J.~Lygeros, and F.~D\"orfler, ``Data-enabled predictive control: In
  the shallows of the {DeePC},'' in \emph{European Control Conf.}, Naples,
  Italy, 2019, pp. 307--312.

\bibitem{coulson2021distributionally}
J.~Coulson, J.~Lygeros, and F.~D{\"o}rfler, ``Distributionally robust chance
  constrained data-enabled predictive control,'' \emph{IEEE Transactions on
  Automatic Control}, vol.~66, no.~9, pp. 4080--4095, 2021.

\bibitem{carlet2020data}
P.~G. Carlet, A.~Favato, S.~Bolognani, and F.~D{\"o}rfler, ``Data-driven
  predictive current control for synchronous motor drives,'' in \emph{IEEE
  Energy Conversion Congress and Exposition}, Detroit, MI, USA, 2020, pp.
  5148--5154.

\bibitem{huang2019data}
B.~Huang, J.~Mohammadi, S.~Riverso, and G.~Ferrari-Trecate, ``Data-driven model
  predictive control for grid-connected power converters,'' in \emph{Proc. IEEE
  Conf. Decision and Control (CDC)}, Nice, France, 2019, pp. 752--757.

\bibitem{berberich2022linear}
J.~Berberich, J.~K{\"o}hler, M.~A. M{\"u}ller, and F.~Allg{\"o}wer, ``Linear
  tracking mpc for nonlinear systems—part {II}: The data-driven case,''
  \emph{IEEE Transactions on Automatic Control}, vol.~67, no.~9, pp.
  4406--4421, 2022.

\bibitem{Lazar2024ECC_BasisFunctionsDeePC}
M.~Lazar, ``Basis-functions nonlinear data-enabled predictive control:
  Consistent and computationally efficient formulations,'' in \emph{European
  Control Conference}, Stockholm, Sweden, 2024, pp. 888--893.

\bibitem{breschi2023datadriven}
V.~Breschi, A.~Chiuso, and S.~Formentin, ``Data-driven predictive control in a
  stochastic setting: a unified framework,'' \emph{Automatica}, vol. 152, p.
  110961, 2023.

\bibitem{Faulwasser2023ARC}
T.~Faulwasser, R.~Ou, G.~Pan, P.~Schmitz, and K.~Worthmann, ``{Behavioral
  theory for stochastic systems? A data-driven journey from Willems to Wiener
  and back again},'' \emph{Annual Reviews in Control}, vol.~55, pp. 92--117,
  2023.

\bibitem{martinelli2022data}
A.~Martinelli, M.~Gargiani, M.~Dra\v{s}kovi\'c, and J.~Lygeros, ``Data-driven
  optimal control of affine systems: A linear programming perspective,''
  \emph{IEEE Control Systems Letters}, vol.~6, pp. 3092--3097, 2022.

\bibitem{vankan2025federated}
G.~Vankan, V.~Breschi, and S.~Formentin, ``Toward federated deepc: borrowing
  data from similar systems,'' arXiv preprint arXiv:2507.17610, 2025, submitted
  23 July 2025.

\bibitem{Markovsky2025Affine}
I.~Markovsky, J.~Eising, and A.~Padoan, ``How to represent and identify affine
  time-invariant systems?'' \emph{IEEE Control Systems Letters}, vol.~9, pp.
  1207--1212, 2025.

\bibitem{Padoan2023DataDriven}
A.~Padoan, F.~D\"orfler, and J.~Lygeros, ``Data-driven representations of
  conical, convex, and affine behaviors,'' in \emph{Proc. 62nd IEEE Conf.
  Decision and Control (CDC)}, Singapore, 2023, pp. 5356--5363.

\bibitem{Padoan2016CDC}
A.~Padoan, G.~Scarciotti, and A.~Astolfi, ``A geometric characterization of the
  persistence of excitation condition for sequences generated by discrete-time
  autonomous systems,'' in \emph{Proc. 55th IEEE Conf. Decision and Control
  (CDC)}, 2016, pp. 3843--3847.

\bibitem{Padoan2016NOLCOS}
------, ``A geometric characterization of persistently exciting signals
  generated by autonomous systems,'' in \emph{Proc. IFAC Symp. Nonlinear
  Control Systems (NOLCOS)}, 2016, available as PDF.

\bibitem{padoan2017geometric}
------, ``A geometric characterization of the persistence of excitation
  condition for the solutions of autonomous systems,'' \emph{IEEE Transactions
  on Automatic Control}, vol.~62, no.~11, pp. 5666--5677, 2017.

\bibitem{Berberich2023ACC}
J.~Berberich, A.~Iannelli, A.~Padoan, J.~Coulson, F.~Dörfler, and
  F.~Allgöwer, ``A quantitative and constructive proof of willems’
  fundamental lemma and its implications,'' in \emph{Proc. American Control
  Conf. (ACC)}, 2023, pp. 4155--4160.

\bibitem{CamlibelRapisarda2024}
\BIBentryALTinterwordspacing
K.~K. Camlibel and P.~Rapisarda, ``Beyond the fundamental lemma: From finite
  time series to linear system,'' \emph{arXiv preprint: arXiv:2405.18962},
  2024. [Online]. Available: \url{https://arxiv.org/abs/2405.18962}
\BIBentrySTDinterwordspacing

\bibitem{Verhoek2023DirectLPV}
C.~Verhoek, H.~S. Abbas, and R.~T{\'o}th, ``Direct data-driven lpv control of
  nonlinear systems,'' \emph{IFAC-PapersOnLine}, vol.~56, no.~2, pp.
  7133--7138, 2023.

\bibitem{lpv-fl}
C.~Verhoek, I.~Markovsky, S.~Haesaert, and R.~Toth, ``The behavioral approach
  for {LPV} data-driven representations,'' \emph{IEEE Trans. Automat. Contr.},
  2025.

\bibitem{lpv-ident}
\BIBentryALTinterwordspacing
I.~Markovsky, C.~Verhoek, and R.~Toth, ``The most powerful unfalsified linear
  parameter-varying model,'' \emph{Automatica}, 2026. [Online]. Available:
  \url{https://imarkovs.github.io/publications/lpv-sysid.pdf}
\BIBentrySTDinterwordspacing

\bibitem{Padoan2018ECC}
A.~Padoan and A.~Astolfi, ``The dimension estimation problem for nonlinear
  systems,'' in \emph{Proc. European Control Conf. (ECC)}, 2018, pp. 184--189.

\bibitem{Berberich2019TrajectoryFramework}
J.~Berberich and F.~Allg{\"o}wer, ``A trajectory-based framework for
  data-driven system analysis and control,'' in \emph{2019 IEEE 58th Conference
  on Decision and Control (CDC)}, 2019, pp. 899--904.

\bibitem{Straesser2021BeyondPoly}
R.~Straesser, J.~Berberich, and F.~Allg{\"o}wer, ``Data-driven control of
  nonlinear systems: Beyond polynomial dynamics,'' in \emph{Proceedings of the
  60th IEEE Conference on Decision and Control (CDC)}, 2021, pp. 1785--1790.

\bibitem{markovsky2008data}
I.~Markovsky and P.~Rapisarda, ``Data-driven simulation and control,''
  \emph{\IJC}, vol.~81, no.~12, pp. 1946--1959, 2008.

\bibitem{vanWaarde2023qmi}
H.~J. van Waarde, M.~K. Camlibel, J.~Eising, and H.~L. Trentelman, ``Quadratic
  matrix inequalities with applications to data-based control,'' \emph{SIAM
  Journal on Control and Optimization}, vol.~61, no.~4, pp. 2251--2281, 2023.

\bibitem{khalil1996nonlinear}
H.~K. Khalil, \emph{Nonlinear systems (3rd edition)}.\hskip 1em plus 0.5em
  minus 0.4em\relax Upper Saddle River, N.J.: Prentice Hall, 1996.

\bibitem{slotine1991applied}
J.-J.~E. Slotine and W.~Li, \emph{Applied Nonlinear Control}.\hskip 1em plus
  0.5em minus 0.4em\relax Prentice Hall, 1991.

\bibitem{isidori1995nonlinear}
A.~Isidori, \emph{Nonlinear Control Systems}, 3rd~ed., ser. Communications and
  Control Engineering Series.\hskip 1em plus 0.5em minus 0.4em\relax
  Springer-Verlag, 1995.

\bibitem{willems1989models}
J.~C. Willems, ``Models for dynamics,'' \emph{Dynamics Reported}, vol.~2, pp.
  171--269, 1989.

\bibitem{willems1997introduction}
J.~C. Willems and J.~W. Polderman, \emph{Introduction to mathematical systems
  theory: a behavioral approach}.\hskip 1em plus 0.5em minus 0.4em\relax New
  York, NY, USA: Springer, 1997.

\bibitem{eisenbud1995commutative}
D.~Eisenbud, \emph{Commutative Algebra with a View Toward Algebraic Geometry},
  ser. Graduate Texts in Mathematics.\hskip 1em plus 0.5em minus 0.4em\relax
  New York: Springer, 1995, vol. 150.

\end{thebibliography}

\end{document}